\documentclass[12pt]{amsart}

\usepackage{amsfonts, amstext, amsmath, amsthm, amscd, amssymb}
\usepackage{epsfig, graphics, psfrag}
\usepackage{color}
\usepackage{url}
\usepackage{adjustbox}

\let\oldmarginpar\marginpar
\renewcommand\marginpar[1]{\oldmarginpar[\raggedleft\footnotesize #1]%
{\raggedright\footnotesize #1}}

\renewcommand{\setminus}{{\smallsetminus}}

\newcommand{\Z}{{\mathbb{Z}}}
\newcommand{\NN}{{\mathbb{N}}}

\newcommand{\cut}{{\backslash \backslash}}

\newcommand{\vol}{{\rm vol}}

\theoremstyle{plain}
\newtheorem{theorem}{Theorem}[section]
\newtheorem{corollary}[theorem]{Corollary}
\newtheorem{lemma}[theorem]{Lemma}
\newtheorem{proposition}[theorem]{Proposition}

\newtheorem{conjecture}[theorem]{Conjecture}

\newtheorem*{namedtheorem}{\theoremname}
\newcommand{\theoremname}{testing}

\theoremstyle{definition}
\newtheorem{definition}[theorem]{Definition}
\newtheorem{remark}[theorem]{Remark}

\title[Cosets of monodromies and quantum representations]{Cosets of monodromies and quantum representations} 
\thanks{Kalfagianni's research was partially supported  by NSF grants DMS-1708249 and DMS-2004155 and a grant from  the Institute for Advanced Study School of Mathematics.}
\thanks{Detcherry's research was  supported  by a postdoctoral fellowship from the Max Planck Institute for Mathematics}

\author{Renaud Detcherry}
\address{Institut de Math\'ematiques de Bourgogne \\
         Faculté des Sciences Mirande, 9 avenue Alain Savary, BP 47870, 21078 Dijon Cedex, France         
         }
\email{renaud.detcherry@u-bourgogne.fr}
\author{Efstratia Kalfagianni}
\address{Department of Mathematics, Michigan State University, East
Lansing, MI, 48824, USA}
\email{kalfagia@math.msu.edu}

\begin{document}

\date{\today}

\begin{abstract} We use geometric methods to show that  given any $3$-manifold $M$, and $g$ a sufficiently large integer,  the mapping class group $\mathrm{Mod}(\Sigma_{g,1})$ contains a coset of an abelian subgroup of rank $\lfloor \frac{g}{2}\rfloor,$ consisting of pseudo-Anosov monodromies of open-book decompositions in $M.$ We prove a similar result for rank two free cosets of $\mathrm{Mod}(\Sigma_{g,1}).$
These results have applications to a conjecture of Andersen, Masbaum and Ueno about quantum representations of surface mapping class groups.
For surfaces with boundary, and large enough genus,  we construct cosets of abelian and free subgroups of their mapping class groups consisting of elements that satisfy the conjecture. 
The mapping tori of these elements are fibered 3-manifolds that satisfy 
a weak form of the Turaev-Viro invariants  volume conjecture.
 \end{abstract}


\maketitle

\section{Introduction}
\label{sec:intro}

It has been known since Alexander \cite{Alexander:open_books} that any closed, oriented $3$-manifold admits an open book decomposition, and
Myers \cite{Myers:open_books} proved that there is one with connected binding. In recent years open book decompositions received attention as they are closely related to contact geometry
through the work of Giroux \cite{Giroux}.

For a compact oriented surface $\Sigma:=\Sigma_{g,n}$ with genus $g$ and $n$ boundary components, let $\mathrm{Mod}(\Sigma)$ denote the  mapping class group of $\Sigma$.
For $f \in \mathrm{Mod}(\Sigma)$ and a subgroup $H$ of  $\mathrm{Mod}(\Sigma)$,
 we will say that $fH$ is a  \textit{rank $k$ abelian coset} if $H$ is abelian of rank $k$. Similarly, we will say that  $fH$ is a  \textit{free coset} if $H$ is free and non-abelian. Moreover in the whole text, we use the term "free group" for free non-abelian group.

Given a closed, oriented  3-manifold $M$, it makes sense to ask how large is the set of mapping classes in $\mathrm{Mod}(\Sigma_{g,1})$
that occur as monodromies
of open book decompositions of $M$. 
One way to quantify this size would be by the rank of abelian subgroups whose cosets occur as monodromies in $M$.
The main result of this article, Theorem \ref{thm:coset} below,
 gives a lower bound on the size. It says that for $g>g_1$, there is at least an abelian  coset of rank $\simeq g/2,$ for any given $M$.
 At this writing we do not know  of an upper-bound of this size other than  the maximal rank of any abelian subgroup of $\mathrm{Mod}(\Sigma_{g,1})$ which is known to be  $3g-2.$ \cite{Tits}

\begin{theorem}\label{thm:coset} Let $M$ be a closed, orientable 3-manifold.
There is an integer $g_1=g_1(M)>0,$ such that for any genus $g\geqslant g_1,$ there is a rank $\lfloor \frac{g}{2}\rfloor$ 
abelian coset of $\mathrm{Mod}(\Sigma_{g,1})$ consisting of pseudo-Anosov mapping classes all of which  occur as monodromies of open book decompositions in $M.$

 Moreover, for any $g\geqslant g_1,$ there is a free non-abelian coset of  $\mathrm{Mod}(\Sigma_{g+4,1})$,  that contains infinitely many pseudo-Anosov mapping classes, consisting  of monodromies of fibered  knots in $M.$
\end{theorem}

It has been long known that any closed, oriented 3-manifold admits open book decompositions  with connected binding and pseudo-Anosov  monodromy.
This is the starting point of our proof of Theorem  \ref{thm:coset}. First,
inspired by a construction of Colin and Honda \cite{ColinHonda}, we show that any open book decomposition with pseudo-Anosov monodromy can be stabilized to one with pseudo-Anosov monodromy and so that the pages of the decomposition  support  several  \textit{Stallings twists} (see Theorem
\ref{stabilize}). To find these stabilizations we use techniques from the study of the geometry of {\textit{curve complexes}} of surfaces.
Other key ingredients of our proof are a refined version of a result of Long-Morton and Fathi  (see Theorem \ref{hyperbolic}) on composing pseudo-Anosov mapping classes with products of powers of Dehn twists, and a result of Hamidi-Tehrani  \cite{HamTeh} on generating free subgroups of mapping class groups.

Given an abelian coset $fH_k$  in $\mathrm{Mod}(\Sigma_{g,1})$ one  can ask  in which 3-manifolds can the elements of $fH_k$ occur as  monodromies of open book decompositions of $M$. In 
this direction we have the following contribution.

 \begin{corollary} \label{boundedintro} Given $k\geqslant 1$, there is a constant $C=C(k)$ with the following property:
 For any  $g\geqslant  7+2k$,  we have a rank $k$ abelian coset 
 $fH_k$   in  $\mathrm{Mod}(\Sigma_{g,1}),$  consisting  of mapping classes that cannot occur as monodromies of open book decompositions in any 3-manifold with Gromov norm larger than $C$.
   \end{corollary}

The author's interest in Theorem \ref{thm:coset} and Corollary \ref{boundedintro} is partly motivated by open questions in quantum topology. Indeed as we will explain next these results have applications
to a conjecture of  Andersen, Masbaum and Ueno \cite{AMU} that relate certain quantum mapping class groups representations to Nielsen-Thurston theory.  The results also provide constructions of fibered manifolds 
that satisfy a weak version of the Turaev-Viro invariants  volume conjecture \cite{Chen-Yang}.

\smallskip

\subsection{ Quantum topology applications}
Given an odd integer $r$,  a primitive $2r$-th root of unity and a coloring $c: |\partial \Sigma | \rightarrow \lbrace 0, 2,\ldots, r-3\rbrace,$
the $\mathrm{SO}(3)$-Witten-Reshetikhin-Turaev TQFT \cite{ BHMV2, ReTu, Turaevbook} gives a
projective representation $$\rho_{r,c} : \mathrm{Mod}(\Sigma) \rightarrow \mathrm{PGL}_{d_{r,c}}(\mathbb{C})$$
called the $\mathrm{SO}(3)$-\textit{quantum representation}.
The dimensions of the representations $d_{r,c}$ can be computed using the so-called Verlinde formula, see for example \cite{ BHMV2}.
 
 The quantum  representations have led to many interesting applications and have turned out to be one of the most fruitful tools coming out of quantum topology \cite{Andersen, GM, KS, MR12}.
However, the geometric content of the representations is still not understood. The AMU conjecture predicts how the Nielsen-Thurston classification of mapping classes should be reflected in the quantum representations.

\begin{conjecture}\label{AMU}{\rm{(AMU conjecture \cite{AMU})}}{  A  mapping class $\phi \in  \mathrm{Mod}(\Sigma)$ has  pseudo-Anosov parts if and only if, for any big enough  $r,$ there is a choice of colors $c$ of the components of $\partial \Sigma,$ such that $\rho_{r,c}(\phi)$ has infinite order.}
\end{conjecture}

To clarify the hypothesis of Conjecture \ref{AMU} recall that, by the  Nielsen-Thurston classification, a mapping class $\phi$ of infinite order is either \textit{ reducible} or   pseudo-Anosov. In the former case, a power of $\phi$ fixes a collection
of disjoint simple closed curves on $\Sigma$ and acts on the components of $S$ cut along these curves. If the induced  map is  pseudo-Anosov
on at least one component, then $\phi$ is said to have  pseudo-Anosov parts. One direction of Conjecture \ref{AMU} is known, that is if $\rho_{r,c}(\phi)$ has infinite order,
then $f$ has pseudo-Anosov parts.
It is also known that the conjecture \ref{AMU} is true for  a mapping class $\phi$ if and only if it is true for  its   pseudo-Anosov parts \cite{AMU, Andersen2}.

 \begin{theorem}\label{thm:fig8example} For any
$k\geqslant 1,$ there is a rank $k$ abelian coset of  $\mathrm{Mod}(\Sigma_{7+2k,1}),$ consisting entirely of mapping classes that satisfy the AMU conjecture.
Furthermore, the coset contains infinitely many pseudo-Anosov mapping classes.
\end{theorem}

In \cite{AMU} Andersen, Masbaum and Ueno  verified the conjecture for $\Sigma_{0,4}.$ Later, Santharoubane proved it for $\Sigma_{1,1}$  \cite{San12}  
and Egsgaard and Jorgensen   \cite{EgsJorgr}
have partial results for pseudo-Anosov maps on $\Sigma_{0,2n}$. In both cases the quantum representations can be asymptotically connected  to
 homological representations, and this connection is the main tool used in those results.  Santharoubane \cite{San17}
gave additional evidence to the conjecture for  $\Sigma_{0,2n},$ by relating the asymptotics of quantum representations to MacMullen's braid group representations.
However, for higher genus surfaces, there is no known connection between quantum and homological representations.
For $g\geqslant 2$,  the first examples of mappings  classes that satisfy the AMU conjecture, were given by March\'e and Santharoubane in \cite{MarSan}. Their method allows the construction of finitely many conjugacy classes of pseudo-Anosov examples in $\mathrm{Mod}(\Sigma_{g,1})$ by choosing suitable  kernel elements of the corresponding Birman exact sequence.
 
 In \cite{DK:AMU} we presented a new approach to Conjecture \ref{AMU} that seems to be the most powerful and promising  strategy currently available.
 Our approach is to relate the AMU conjecture
 to a weak  version of the volume conjecture stated by Chen and Yang \cite{Chen-Yang} which we first addressed in \cite{DK:Volume}.
 Using this approach, we gave  the first infinite families of pseudo-Anosov examples that satisfy the AMU conjecture. More specifically, we showed that, for
 $g\geqslant n\geqslant 3$ or $n=2$ and $g\geqslant 3,$ in $\mathrm{Mod}(\Sigma_{g,n})$ there exists infinitely many non-conjugate, non power of each other, pseudo-Anosov maps that satisfy the conjecture.
Kumar \cite{Sanjay}, also using the approach of   \cite{DK:AMU}, constructed infinite families of such examples in $\mathrm{Mod}(\Sigma_{g,4})$, for 
for $g>2$. Moreover, for any $n>8$,  he gave explicit families of elements  in  $\Sigma_{0, n}$ that satisfy the AMU conjecture.
 
 In this paper, we turn our attention to the mapping class group $\mathrm{Mod}(\Sigma_{g,1})$ of compact surfaces with one boundary component, and construct entire cosets of abelian and free subgroups that satisfy Conjecture \ref{AMU}. Our method for proving Theorem \ref{thm:fig8example} is flexible and can be adapted to construct similar cosets for surfaces with more boundary components.

 To describe the approach of \cite{DK:AMU} in more detail,
 given a manifold $M,$ let $TV_r(M)$ be the $\mathrm{SO}(3)$-Turaev-Viro invariant of $M$ at odd  $r\geq 3$ and root of unity $q=e^{\frac{2i\pi}{r}}.$ Also, let us define the \textit{Turaev-Viro limit} $lTV(M)$ by 
\begin{equation}
lTV(M)=\underset{r\rightarrow \infty, r \ \textrm{odd}}{\liminf} \frac{2\pi}{r}\log TV_r(M)
\end{equation}

To facilitate our exposition, we  define   a compact oriented $3$-manifold $M$ to be  \textit{q-hyperbolic} iff
we have  $lTV(M)>0.$
 
An important open problem in quantum topology is the volume conjecture of  \cite{Chen-Yang} asserting that  for any finite volume hyperbolic 3-manifold $M$ we have $lTV(M)=\mathrm{Vol}(M),$  which in particular implies that $M$ is q-hyperbolic.
A perhaps more robust conjecture, supported by the computations of \cite{Chen-Yang} and  the results of  \cite{ BDKY, D:cabling,  DK:Volume, DKY}, is the following.
\begin{conjecture}\label{EGC}{ \rm{ (Exponential growth conjecture)}} \label{EGC} For $\Sigma$ a compact orientable surface and $f\in \mathrm{Mod}(\Sigma),$  let $M_f$ denote the mapping torus of $f$ and let $||M_f||$ denote
the Gromov norm of $M_f$.
Then, $M_f$ is q-hyperbolic if and only if $||M_f||>0.$ 
\end{conjecture}

By the work of Thurston, $f$ has non-trivial  pseudo-Anosov parts if and only if the JSJ-decomposition  of the mapping torus $M_f$ contains hyperbolic parts or equivalently  iff we have $||M_f||>0.$ 
In particular,  $M_f$ is hyperbolic  precisely when  $f$ is pseudo-Anosov, and in this case Conjecture \ref{EGC} is implied by the volume conjecture of \cite{Chen-Yang}.

A result of the authors \cite[Theorem 1.1]{DK:Volume} implies that for every $f\in \mathrm{Mod}(\Sigma),$ we have  $lTV(M_f)\leqslant C \cdot  ||M_f||,$
for some universal constant $C>0$. This in turn implies that if $M_f$ is q-hyperbolic,  then we have  $||M_f||>0,$
which gives 
 the  ``if" direction of Conjecture \ref{EGC}. 
On the other hand, in \cite{DK:AMU} we also proved  that  if $M_f$ is q-hyperbolic, then  $f$ satisfies the AMU conjecture. Hence
Conjecture \ref{EGC} implies Conjecture \ref{AMU}. 

Although the  Chen-Yang volume conjecture  is wide open, there are vast families of q-hyperbolic manifolds and of manifolds that satisfy Conjecture \ref{EGC}. For example,  we have families of closed  q-hyperbolic 3-manifolds of arbitrarily large Gromov norm  \cite{DK:Volume} and we know that 3-manifolds obtained by drilling out links from q-hyperbolic manifolds are also q-hyperbolic  \cite{DK:Volume}.
A more detailed list of $3$-manifolds known to be q-hyperbolic will be described in Section \ref{sec:q-hyperbolic}.
Applying Theorem \ref{thm:coset} to q-hyperbolic manifolds $M,$ we have the following.
\begin{corollary}\label{apAMU} Suppose that $M$  is q-hyperbolic. Then the mapping classes in the abelian and free cosets given by Theorem  \ref{thm:coset} satisfy the AMU conjecture and the corresponding mapping tori are q-hyperbolic.
\end{corollary}

Note that the  method of our proof of Theorem \ref{thm:coset}, combined with hyperbolic Dehn filling techniques,  allows  the construction of pairs of cosets that satisfy Conjecture \ref{AMU} and  are independent in the sense that
no mapping class in one coset is a conjugate of a mapping class in the other. See Corollary \ref{independent}.
\smallskip

\subsection{Organization}  In Section \ref{sec:prelim} we lay out known definitions and results that we will use in the remaining of the paper. In Section \ref{sec:filling} we define a notion of independence of mapping classes 
and we give the proof of a refined version of a result of Long-Morton and Fathi  that we need for the proof of Theorem \ref{thm:coset}. We also discuss the existence of infinitely many
independent  pseudo-Anosov classes in certain mapping class group cosets (Corollary \ref{infinitelymany}). In Section \ref{sec:stabilization} we prove Theorem \ref{thm:coset}  and Corollary \ref{apAMU}.
In Section \ref{sec:fig8} we give constructions of fibered knots in certain 3-manifolds obtained by surgery on the figure eight knot. The main result of the section is Theorem \ref{thm:fig8examplegeneral} which, in particular, implies Corollary \ref{boundedintro} and Theorem \ref{thm:fig8example}.

We note that, although our results here are partly motivated by questions in quantum topology, the techniques and tools we use are from geometric topology and hyperbolic geometry. In particular, one needs no additional knowledge of quantum topology, than what  is given in this Introduction, to read the paper.

\vskip 0.08in

{\bf{Acknowledgements.}} We thank J. Andersen, Matt Hedden, Slava Krushkal, Feng Luo and Filip Misev for their interest in this work and for useful conversations. We also thank  Dave Futer for bringing to our attention the article \cite{ HamTeh}.  This work is based on research done while  Kalfagianni was on sabbatical leave from MSU and supported by NSF grants DMS-1708249 and DMS-2004155 and a grant from  the Institute for Advanced Study School of Mathematics.  Detcherry was supported by the Max Planck Institute for Mathematics during this work and thanks the institute for its hospitality and support.


\section{Preliminaries}
\label{sec:prelim}

In this section we summarize  some definitions and results that we will use in this paper.

\subsection{Open book decompositions, fibrations and Stallings twist} \label{subsec:Stallings}
We start by describing a more $2$-dimensional way of thinking of fibered links in a closed compact oriented $3$-manifold $M,$ or equivalently, open book decompositions of $M.$
We will keep track of an open book decomposition of a $3$-manifold $M$ as a pair $(\Sigma,h),$ where $\Sigma$ is the fiber surface and $h$ is the monodromy. 
The  mapping torus $M_{h}=\Sigma \times [0,1]/_{(x,1)\sim ({ {h}(x),0)}}$, with fiber $\Sigma:= \Sigma \times \{0\}$, is homeomorphic to a link complement in $M$.
The 3-manifold $M$ is determined, up to homeomorphism, as the relative mapping torus of  $h$:  that is, $M$ is homeomorphic to the quotient of $M_h$
under the identification
$(x,t)\sim (x, t'), \ \ {\rm for \ all}\ \  x\in\partial \Sigma, t,t'\in [0, 1].$
The quotient of $\partial M_h$ under this identification, is the fibered link $K$,  called the binding of the open-book decomposition.
We will slightly abuse the setting and consider $M_f$ embedded in $M$ so that it is the complement of a neighborhood of $K$.
It is known since Alexander that any compact oriented $3$-manifold has an open book decomposition and work of Myers \cite{Myers:open_books} shows that  there is one with connected binding and pseudo-Anosov monodromy.

Suppose that there is a homotopically non-trivial curve $c$  on $\Sigma$ that bounds an embedded disk $D$ in $M$, which intersects $\Sigma$ transversally, such that we have
$lk(c,c^+)=0,$ where $c^+$ is the curve $c$ pushed along the normal of $\Sigma$ in $\Sigma \times \{ t\}$, for small $t>0$. We will refer to this as pushing off in the positive direction.
The disk $D$ can be enclosed in a 3-ball $B$ that intersects $\Sigma$ transversally. In this ball we can perform 
a  twist of order $m$ along $D.$ 
\begin{definition} \label{def:Stallingstwist}  This operation above is called a \textit{Stallings twist}  of order $m$, along $c$. We will say that the curve $c$  supports a Stallings twist.
A curve that supports a Stallings twist will be called a \textit{Stallings curve}.
\end{definition}

We will consider $M, \Sigma$ and $K$ oriented so that the orientation of $\Sigma$ (resp. $K$) is induced by that of $M$ (resp. $K$).
A Stallings twist can also be thought of as performing $1/m$ surgery along the unknotted curve  $c$ inside $B$, where the framing of $c$ is induced by  $\Sigma$ on $c$.
This operation will not change the ambient manifold $M$ but it changes $M_f$ to a mapping torus with fiber $\Sigma$ and monodromy $f\circ \tau_{c}^m$. See \cite{Stallings} for more details.
\smallskip

\subsection{Families of q-hyperbolic $3$-manifolds}
\label{sec:q-hyperbolic}
We recall from Section \ref{sec:intro} that a q-hyperbolic $3$-manifold was defined to be a $3$-manifold $M$ such that $lTV(M)>0.$ Such manifolds will be the starting pieces in our constructions in Section \ref{sec:fig8} and Section \ref{sec:stabilization}.

The following theorem will sum up the known examples of q-hyperbolic manifolds:

\begin{theorem}\label{thm:q-hyperbolic}The following compact oriented $3$-manifolds are q-hyperbolic:
\begin{enumerate}
\item The figure-eight knot and the Borromean rings complements in $S^3$.
\item  Any 3-manifold obtained by Dehn surgery on the figure-eight knot with slopes  determined by integers $p$ with $|p|\geqslant 5$.
\item The fundamental shadow-link complements of \cite{CoThu}. This is an infinite family of cusped hyperbolic manifolds
with arbitrarily large volumes.
\item Suppose that $L$ is a link in $S^3$ with q-hyperbolic complement. Then any link obtained from $L$ by replacing a component with any number of parallel copies or by a $(2n+1,2)$-cable 
has q-hyperbolic complement.
\item Any $3$-manifold $M'$ that can be obtained from a q-hyperbolic manifold $M$  by drilling solid tori.  In fact we have  $LTV(M')\geqslant LTV(M)>0$.
\item Any  link complement in any q-hyperbolic 3-manifold.
\item Any closed orientable  3-manifold $DM$ obtained by gluing two oppositely oriented copies of a q-hyperbolic 3-manifold $M$ with toroidal boundary along $\partial M.$
\end{enumerate}
\end{theorem} 
\begin{proof} The Turaev-Viro invariants volume conjecture is known for the classes of manifolds in (1)-(3) above: For figure-eight knot and the Borromean rings complements 
the conjecture was proved by Detcherry-Kalfagianni-Yang \cite{DKY}, for the integral surgeries on  the figure-eight knot  by Ohtsuki  \cite{Ohtsuki:figure8} and for fundamental shadow link complements by Belletti-Detcherry-Kalfagianni-Yang \cite{BDKY}.
Part (4) follows by \cite[Corollary 8.4]{DK:Volume} and by the main result of
\cite{D:cabling}.  Part (5) follows from \cite[Corollary 5.3]{DK:Volume} and part (6) is just a reformulation of part (5). Finally, part (7) follows  by \cite[Theorem 3.1 and Remark 3.4]{DKY}.
\end{proof}

As a corollary of Theorem \ref{thm:q-hyperbolic} we have the following.

\begin{corollary} \label{large} There are closed  $q$-hyperbolic 3-manifolds of arbitrary large Gromov norm.
\end{corollary}
\begin{proof} Given an oriented $q$-hyperbolic 3-manifold $M$ with toroidal boundary
let $D(M)$ denote the double. By Theorem \ref{thm:q-hyperbolic},  it follows that $D(M)$ is $q$-hyperbolic. On the other hand, the doubles of  fundamental shadow-link complements have arbitrarily large Gromov norm.
\end{proof}

Finally, we need to recall the following theorem.

\begin{theorem}\label{tool} Let $M$ be a closed q-hyperbolic 3-manifold. Then any mapping class $f\in \mathrm{Mod}(\Sigma_{g,n})$
that occurs as monodromy of a fibered link in $M$ satisfies the AMU conjecture.
\end{theorem}
\begin{proof} Let $L$ be an $n$-component fibered link in $M$. By Theorem  \ref{thm:q-hyperbolic} (4), the complement of $L$ is q-hyperbolic.
On the other hand, if we choose a fibration of this complement and a monodromy  $f$,  we can realize it  as a mapping torus $M_{f}=\Sigma_{g,n} \times [0,1]/_{(x,1)\sim ({ {f}(x),0)}},$
where $g$ is the genus of the fiber. Now the result follows by  \cite[Theorem 1.2]{DK:AMU}\footnote{Note that although \cite[Theorem 1.2]{DK:AMU} is stated for pseudo-Anosov maps,
the proof given does not use this hypothesis.}
\end{proof}

\subsection{ Abelian and free subgroups of mapping class groups}
It is known that the mapping class group $ \mathrm{Mod}(\Sigma)$ of any compact, connected,  oriented surface $\Sigma=\Sigma_{g,n}$ satisfies a ``Tits alternative property". That is, any subgroup of 
$ \mathrm{Mod}(\Sigma)$ either contains an abelian group of finite index or it contains a free non-abelian group \cite{Tits}.  Thus abelian and free subgroups are abundant in $ \mathrm{Mod}(\Sigma)$.
Furthermore it is known that any abelian subgroup of
$ \mathrm{Mod}(\Sigma)$ is finitely generated and has rank at most $3g-3+n$ \cite{BLM}.

Let $\Gamma$ be a set of $k$ essential,  simple,  closed, disjoint curves on $\Sigma$ that are non-boundary parallel and no pair of which is parallel to each other. Then the Dehn twists along the curves 
in $\Gamma$  generate a free abelian subgroup of rank $k$.  All the abelian subgroups  of $ \mathrm{Mod}(\Sigma)$ that we will consider in this paper are of this type.
To produce free subgroups we will use the following theorem of Hamidi-Tehrani.
\begin{theorem}\label{Hamidi-Tehrani} {\rm{ (\cite[Theorem 3.5]{HamTeh})}}  Let $a,b$  be essential, simple closed curves on $\Sigma$ and let $i(a, b)$ denote their intersection number.  Let $\tau_a, \tau_b$ denote the positive Dehn twists along $a,b$, respectively. Suppose that $i(a, b)\geqslant 2$.
 Then, the subgroup $\left<\tau_a, \tau_b\right> <\mathrm{Mod}(\Sigma)$ generated by $\tau_a$ and $\tau_b$
 is a free group of rank two.
\end{theorem}

\smallskip

\section{Dehn filling and cosets of monodromies}
\label{sec:filling}

\subsection{Independance of mapping classes}
\label{sec:independance}
 We start this section by clarifying the notion of independent mapping classes we mentioned in Section \ref{sec:intro}.
\begin{definition}\label{def:independent}Let $\Sigma$ be a compact oriented surface and let  $f,g \in \mathrm{Mod}(\Sigma)$. We say that $f$ and $g$ are independent if there is no $h\in \mathrm{Mod}(\Sigma)$ so that both $f$ and $g$ are conjugated to non-trivial powers of $h.$
\end{definition}

\begin{remark} Definition \ref{def:independent} is motivated by the following observation: if $f$ and $g$ are not independent, then $f$ satisfies the AMU conjecture if and only if $g$ does. Indeed, let $f$ be conjugated to $h^k$ and $g$ to $h^l,$ for some integers $k,l\neq 0.$ For any representation $\rho$ of $\mathrm{Mod}(\Sigma),$ the images $\rho(f)$ and $\rho(g)$ will have infinite order if and only if $\rho(h)$ does.
\end{remark}

With the following lemma we can show the independence of two elements using the volumes of the corresponding mapping tori.
\begin{lemma}\label{lemma:independent} Let $\Sigma$ be a compact oriented surface. For $f \in \mathrm{Mod}(\Sigma)$ let $M_f$ be the mapping torus of $f.$ There is a constant $\epsilon>0$ such that for any pseudo-Anosov elements $f,g \in \mathrm{Mod}(\Sigma),$ if $\vol(M_f)\neq {\vol}(M_g)$ and $|\vol(M_f)-{\vol}(M_g)|<\epsilon$ then $f$ and $g$ are independent.
\end{lemma}
\begin{proof}
If $f$ is conjugated to $h^k$ and $g$ to $h^l$ for some $h\in \mathrm{Mod}(\Sigma)$ and $k,l\in \Z \setminus \lbrace 0 \rbrace,$ then ${\vol}(M_f)=k{\vol}(M_h)$ and $\vol(M_g)=l\vol(M_h).$ As $M_f$ and $M_g$ have different volumes, their volume differ at least by ${\vol}(M_h).$ By Jorgensen-Thurston \cite{thurston:notes} theory the set of volumes of hyperbolic manifolds is well ordered, thus it has a minimal element.
Thus we can simply take $\epsilon$ to be smaller than the minimal volume.
\end{proof}

We need the following theorem which for $k=1$ is proved by  Long and Morton \cite[Theorem 1.2]{LongMorton}. 

 \begin{theorem} \label{hyperbolic} Let $f\in \mathrm{Mod}(\Sigma)$ is a pseudo-Anosov map and let $\gamma_1,\ldots, \gamma_k$ be homotopically essential, non-boundary 
 parallel curves. Suppose moreover that no pair of the curves are parallel on $\Sigma$ and that   $i( f(\gamma_i), \gamma_j)\neq 0$, for any $1\leqslant i, j\leqslant k$.
 Then, there is $n\in \NN$ such that for all
 $(n_1, \ldots, n_k)\in {\Z}^k$ with all $|n_i|>n,$  the mapping class
 $$g=f\circ \tau^{n_1}_{\gamma_1} \circ \ldots \circ \tau^{n_k}_{\gamma_k}$$ is pseudo-Anosov. Furthermore, the family 
 
 $$\{f \circ \tau^{n_1}_{\gamma_1} \circ  \ldots \circ \tau^{n_k}_{\gamma_k} , \ \ {\rm with} \ \ {(n_1, \ldots, n_k)\in {\Z}^k} \ \ {\rm and } \  \ |n_i|>n\}$$
 contains infinitely many, pairwise  independent pseudo-Anosov mapping classes.
  \end{theorem}
  The argument we give below follows the line of the proof given 
in \cite{LongMorton}, however for $k>1$ we need to add hypotheses in order to assure hyperbolicity of the link complement $M_f \setminus \left( \gamma_1 \cup \ldots \cup \gamma_k \right).$ The statement for $k>1$ is used  in the proof of  Theorem \ref{thm:coset} in Section \ref{sec:stabilization}.
  \begin{proof} Consider the mapping torus
 $$M_{f}=\Sigma \times [0,1]/_{(x,1)\sim ({ {f}(x),0)}},$$
 which is hyperbolic since $f$ is  pseudo-Anosov \cite{thurston:notes}.
 Now take numbers  $1/2=i_1 < i_2< \ldots  <i_k <1$ and consider the curves $\gamma_1,  \ldots,   \gamma_k$ lying on the fiber 
 $F:= \Sigma \times \{ i_1\}\subset M_{f}$.
 For $s=2, \ldots, k$, consider the annulus $\gamma_s \times [i_1, \  i_s]$ from $\gamma_s$ on $F$ to  $\delta_s:=\gamma_s \times \{ i_s \}$ which is a curve lying on the fiber
 $F_{i_s}:= \Sigma \times \{i_s \}\subset M_{f}$.  Let $\delta_1:=\gamma_1$.
 Now $L=\delta_1\cup \ldots  \cup  \delta_k$ is a link in $M_{f}$. 
 \vskip 0.07in
 
\noindent  {\bf Claim.} The manifold $N:=M_{h}\setminus L$ is hyperbolic.
 
  \vskip 0.07in

 \noindent  {\bf Proof of Claim.}  Since the curves $\gamma_1,\ldots, \gamma_k$ are homotopically essential and $M_{f}$ is hyperbolic,  the manifold $N$ is irreducible and boundary irreducible. We need to show that $N$ contains no essential embedded tori. Let $T$ be a torus embedded in $N$.
If $T$ is boundary parallel in $M_{f}$ it will also be in $N$ otherwise one would be able to isotope some of the $\gamma_i$ to be boundary parallel in $\Sigma$ contradicting our assumptions.
The torus $T$ must intersect some of the fibers where the curves $\delta_i$ lie since otherwise it would compress in the handlebody outside of those fibers.  
Suppose, without loss of generality, it intersects
 $F$. Then after isotopy, $F\cap T$ is a collection of  simple closed curves that are homotopically essential in $F\setminus \delta_1$.
The intersection of $T$ with the handlebody $M_F$ is a collection of properly embedded annuli in the handlebody $M_F:=M_{f}\cut F$.
After isotopy every annulus in  $M_F$ is either vertical with respect to the $I$-product or it has both its boundary components on the same copy of $F$ on $\partial M_F$
in which case the components are parallel on $F$. Annuli of the later type can be eliminated by isotopy in $N$ unless  they intersect some fiber $F_{i_s}$ containing a curve $\delta_s$, in which case
the intersection of the annulus with $F_s$ is two parallel curves running on opposite sides of $\delta_s$. 

Suppose now that $A$ is a component of $T\cap M_F$ that is vertical with respect to the $I$-product of $M_F$
and let $\theta$ denote a component of $F\cap T$.  There are two cases to consider.
\vskip  0.06in

{\it Case 1.} The vertical annulus  $A$ connects to an annulus $A'$ of $M_F\cap T$ such that the two components of $\partial A'$ run  parallel to some $\delta_j$ on opposite sides of $\delta_j$ on the fiber that contains it. The vertical annulus runs from $\theta,$ which must be parallel to $\delta_j$ now, to $f(\theta)$ on the other copy of the fiber on $\partial M_F$.
Since we assumed that the original curves are pairwise non-parallel on $\Sigma$,  and that  the intersection number  $f(\gamma_i)$ and $\gamma_j$, is non-zero for any $1\leqslant i, j\leqslant k$, the curve $\theta$ cannot be parallel to any of the curves $\delta_1, \ldots \delta_s$. 
It follows that the torus $T$ cannot contain a second annulus whose boundary curves run parallel to some $\delta_s\neq \delta_j$
on a fiber. Thus $T$ must be boundary parallel in $N$.

\vskip  0.06in

{\it Case 2.}  The vertical annulus  $A$ eventually connects to the curve $\theta$.
Since  $f$ is pseudo-Anosov, and $\theta$ is essential on $F$,  we cannot have
$f^{k}(\theta)=\theta$, for any integer $k\neq 0$. In fact it is known that the intersection number of $\theta$ with  $f^{k}(\theta)$ grows linearly in $k$. See Proposition \ref{diameter}
and Remark \ref{morediameter} below for the precise statements and references.
It follows that such an annulus  cannot close to become a torus in $M_f$.
 Thus, this case will not happen.  This implies that $N$ is atoroidal. 
  \vskip  0.06in

Now an atoroidal manifold, with toroidal boundary, can only contain essential annuli if it is a Seifert fibered manifold. But since $M_f$ is hyperbolic it has non-zero Gromov norm and 
 since this norm doesn't increase under filling the cusps with tori, we conclude that
  $||N||\geqslant ||M_{f}||>0$, which implies that
  it cannot be a Seifert fibered space. Thus $N$ cannot have essential annuli and thus, by Thurston's work $N$ is hyperbolic .
  \vskip 0.05in

 To continue we use the fact that  the mapping torus of $g=f \circ \tau^{n_1}_{\gamma_1} \circ  \tau^{n_k}_{\gamma_k}$
 is obtained by $N$ by doing Dehn filling along slopes say $s_1, \ldots, s_k$, respectively. The length of the slope $s_i$ on the corresponding cusp of $N$ is an increasing function of $|n_j|.$ 
Let $\lambda$ denote the length the shortest slope.
By Thurston's hyperbolic Dehn Filling theorem \cite{thurston:notes} there is $n\in \NN$ such that for all
 ${\bf n}:=(n_1, \ldots, n_k)\in {\Z}^k$ with all $|n_i|>n$, the resulting manifold $N({\bf n})$ obtained by filling $N$, where the cusp corresponding to $\delta_i$ is filled as above, is hyperbolic. The volumes of the filled manifolds $N({\bf n})$ approach the volume of $N$ from below.
 To make things more concrete we use an effective form of the theorem proved in \cite[Theorem 1.1]{fkp:filling}, which states that provided that $\lambda >2\pi$, then  $N({\bf n})$
 is hyperbolic and we have

 $$ \left(1-\left(\frac{2\pi}{\lambda}\right)^2\right)^{3/2}\vol(N)  \leqslant\ \vol(N({\bf n})) )  \leqslant\  \vol(N).$$
 
 Using this it becomes clear than one can find infinite sequences of $k$-tuples $\{{\bf n}_i \}_{i\in\NN}$ such that
 the manifolds $N({\bf n}_i)$ have strictly increasing volumes and  such that for any $i\neq j \in \NN$ we have $|\vol(N({\bf n}_j))-\vol(N({\bf n}_i))|< \epsilon$,
 where $\epsilon$ is the constant of Lemma \ref{lemma:independent}.
 Now the monodromies of the manifolds $N({\bf n}_i)$  above form an infinite sequence of independent pseudo-Anosov mapping classes. \end{proof} 
\smallskip

 Note that if $k=1$
in the statement of Theorem \ref{hyperbolic}, then in fact Fathi has shown that there can be at most seven integers $n_1\in \Z$ for which $f \circ \tau^{n_1}$ is not pseudo-Anosov.
We also note that for the following corollary, which is  \cite[Theorem 1.3]{LongMorton},
we do not need the extra hypotheses on the curves $\gamma_i$ given in the statement of Theorem \ref{hyperbolic}.

\begin{corollary}\label{more} Let $f\in \mathrm{Mod}(\Sigma)$ be a pseudo-Anosov map and let $\gamma_1,\ldots, \gamma_k$ be homotopically essential, non-boundary 
 parallel curves. Then, for any $k>0$, the family of maps  $g=f\circ \tau^{n_1}_{\gamma_1} \circ \ldots \circ \tau^{n_k}_{\gamma_k}$, with $(n_1, \ldots, n_k)\in {\Z-\{0\}}^k$,
contains infinitely many  pseudo-Anosov elements.
\end{corollary} 
\begin{proof}  By Theorem \ref{hyperbolic}, the family of maps $\{f\circ \tau^{n_1}_{\gamma_1} \}_{n_1\in \NN}$, contains infinitely many pseudo-Anosov elements.
Again by  Theorem \ref{hyperbolic}, for each of these pseudo-Anosov elements, the family $\{f\circ \tau^{n_1}_{\gamma_1}\circ  \tau^{n_2}_{\gamma_2}\}_{n_2\in \NN}$ contains infinitely many pseudo-Anosov elements.
That is given $m_1\geq n_1$, there is some $m_2$, depending on $m_1$, such that for any $|n_2|>m_2$,  the map $f\circ \tau^{m_1}_{\gamma_1}\circ  \tau^{n_2}_{\gamma_2}$
is pseudo-Anosov. Continuing inductively we arrive at the desired conclusion.
\end{proof}

\smallskip

\subsection{Abelian and free elementary cosets} To prove Theorems \ref{thm:fig8example} and \ref{thm:coset}, we will construct cosets of abelian and free subgroups of $\mathrm{Mod}(\Sigma)$ generated by Dehn twists on curves as in Section \ref{sec:prelim}. To facilitate clarity of exposition let us make the following definition.
\begin{definition}  Let  $\Sigma$  a compact oriented surface and $f \in \mathrm{Mod}(\Sigma).$
\begin{itemize}
\item If $H$  is a subgroup of $\mathrm{Mod}(\Sigma)$ generated by $k$ (powers of)  Dehn twists along disjoint curves on $\Sigma$, we say that $fH$ is a  \textit{ rank $k$ abelian elementary coset}.
\item If $H$ is a rank two free subgroup of $\mathrm{Mod}(\Sigma)$ generated by (powers of)  Dehn twists  along two curves on $\Sigma$
we say that $fH$ is a {\textit  {free elementary coset}}.
\end{itemize}
Furthermore, in both cases, if $fH$ contains a pseudo-Anosov element, we say that $fH$ is a \textit{pseudo-Anosov elementary coset}.
\end{definition}

An immediate corollary of Theorem \ref{hyperbolic} is the following.

\begin{corollary} \label{infinitelymany} Any pseudo-Anosov abelian or free elementary  coset  of  $\mathrm{Mod}(\Sigma)$
 contains infinitely many
independent pseudo-Anosov  mapping classes.  
\end{corollary}
\begin{proof} Given a pseudo-Anosov  abelian or free elementary coset  $gH$ of $\mathrm{Mod}(\Sigma),$
we can write $fH=gH$, where $f$ is a pseudo-Anosov mapping class. Recall that in both cases, $H$ contains a free abelian group generated by a Dehn twist $\tau_{\gamma}$ on
a simple closed non non-peripheral curve $\gamma$.
 We can look at mapping classes of the form $f\circ \tau_{\gamma}^n$ and apply
 Theorem \ref{hyperbolic}.
\end{proof}

\smallskip


\smallskip
\section{ Abelian and free elementary cosets from 3-manifolds}
\label{sec:stabilization}
The main results in this section are Theorem \ref{prop:stab_abelian}  and Theorem \ref{prop:stab_free}
which imply Theorem \ref{thm:coset}  and Corollary \ref{apAMU} stated in the Introduction.
To prove these results,
first we show that any open book decomposition  of a 3-manifold, with pseudo-Anosov monodromy,   can be stabilized to one with  pseudo-Anosov monodromy and so that the pages of the decomposition  support   Stallings twists. See Theorem
\ref{stabilize} and subsequent discussion. Roughly speaking, the key point is to choose stabilization arcs on the fiber surface  that are parts of curves that are ``complicated" in the sense of the geometry of the  curve complex
$\Sigma_{g,1}$. See Lemma \ref{Pennerlarge} and Propositions \ref{oneboundary} and \ref{twoboundary}. 
To prove Theorems  \ref{prop:stab_abelian}  and  \ref{prop:stab_free} we combine Theorem
\ref{stabilize} with
Theorems \ref{hyperbolic} and \ref{Hamidi-Tehrani}, respectively.

\subsection{Moves on open book decompositions} 
\label{sec:moves_open_book}
An open book decomposition $(\Sigma,h)$ of $M$ can be modified into another open book decomposition of $M$ by some elementary moves called stabilizations and destabilizations:

\begin{definition}\label{def:stabilization} Let $(\Sigma,h)$ be an open book decomposition of  $M,$
and  $(\gamma, \partial \gamma) \subset ( \Sigma, \partial \Sigma)$ a properly embedded arc. We will say that the surface $\Sigma'$ is obtained by a  \textit{positive} (resp.  \textit{negative})  \textit{stabilization} along
$\gamma$, if $\Sigma'$ is obtained from $\Sigma$ by plumbing a positive (resp. negative) Hopf band $H$ (resp. $H^{-1}$) along $\gamma$.  The intersection of the $H^{\pm1 }$ with $\Sigma$ is a neighborhood of
$\gamma$ and the arc $\gamma$ becomes part of the core curve, say $c$, of $H^{\pm1 }$. The curve $c$  bounds a disk in $M,$ whose intersection with $\Sigma$ is exactly the arc $\gamma.$

Let  $\tau_{c}$ denote the corresponding Dehn twist on $c\subset \Sigma'$ and let $h$ be extended to $\Sigma'$ by the identity on $\Sigma'\setminus \Sigma$. 
Then the pair $(\Sigma',h')$ where $h'=h \circ \tau_{c}^{\pm 1}$ is an open book decomposition of $M$ which we call a \textit{stabilization} of $(\Sigma,h).$ 

Finally, $(\Sigma,h)$ is a \textit{destabilization} of $(\Sigma',h').$
\end{definition}

Harer proved \cite{Har82} that any two open book decompositions of the same $3$-manifold $M$ are related by stabilizations, destabilizations and a third elementary move called a double twist move. Harer conjectured that stabilizations and destabilizations are sufficient, which was later proved by Giroux and Goodman \cite{Giroux:open_book}.
 Despite Giroux and Goodman's result, it is in practice difficult to explicitly express a double twist move in terms of stabilizations and destabilizations. Some particular examples of double twist moves are the Stallings twists, which we already introduced in Definition \ref{def:Stallingstwist}.
\subsection{The curve graph and Penner systems} A compact oriented surface $\Sigma=\Sigma_{g,n}$ is called non-sporadic if $3g-3+n>0$. Given a non-sporadic surface $\Sigma$, the \textit{curve graph}
\label{subsec:curves}
  ${\mathcal C}_1(\Sigma)$  is defined as follows:
  
The set of vertices, denoted by   ${\mathcal C}_0(\Sigma)$, is the set of isotopy classes of  homotopically essential, non-boundary parallel simple closed curves and two vertices
   $a, b\in {\mathcal C}_0(\Sigma)$ are connected by an edge if they can be represented by disjoint curves, that is they have intersection number zero ($i(a, b)=0$). The space  ${\mathcal C}_1(\Sigma)$
   becomes a geodesic metric space with the path metric that assigns length 1 to each edge of the graph, and elements of the mapping class of $\Sigma$ act as isometries on the space.
  Now for the \textit{curve complex} ${\mathcal C}(\Sigma)$ one adds to ${\mathcal C}_0(\Sigma)$ cells to collections of disjoint curves on $\Sigma$.

For   $a, b\in {\mathcal C}_0(\Sigma)$  the distance between $a,b$ will be denoted by $d(a,\  b)$.
The geometry of the space was studied by Masur and Minsky \cite{MaMi}.

\begin{lemma} \label{intersection}  \cite[Lemma 2.1]{MaMi} Given  $a, b\in {\mathcal C}_0(\Sigma)$ we have $d(a, \ b) \leqslant 2 i(a, \ b)+1$
\end{lemma}

\begin{proposition} \label{diameter} {\rm \cite[Proposition 4.6]{MaMi}} For any non-sporadic surface $\Sigma$ there is a constant $C>0$ such that given a pseudo-Anosov
$h\in   \mathrm{Mod}(\Sigma)$ we have 
$$d(a, h^n(a)) \geqslant C |n|,$$
for every $a \in {\mathcal C}_0(\Sigma)$ and $n\in \Z$.
\end{proposition}

\begin{remark}\label{morediameter} Proposition \ref{diameter} implies that, if $h$ is pseudo-Anosov  and we have  $h^k(a)=a$, for some $a\in {\mathcal C}_0(\Sigma)$, then $k=0$. For, if $ k \neq 0$, then we  have infinitely many  $n\in \Z$
such $h^n(a)=a$, namely all the multiples of $k$.
Thus we  have $d(a, \ h^n(a)) \leqslant 1$, for infinitely many  $ n\in \Z$. On the other hand, by Proposition  \ref{diameter} we have $d(a,  h^n(a)) \geqslant C |n|,$ which is a contradiction.
\end{remark}

We will need the following lemma.

\begin{lemma} \label{fill} Let $a, b\in {\mathcal C}_0(\Sigma)$  and  let $h\in   \mathrm{Mod}(\Sigma)$ be pseudo-Anosov. Suppose that  
we have $d(a, h(a)) \geqslant N_0$ and $d(b,h(b)) \geqslant N_0$, for some $N_0\gg2i(a, \ b) +1.$ Then, we have $h(a)\neq b$.
\end{lemma}

\begin{proof} We have 
\begin{eqnarray*}d(b,h(a)) & \geqslant &d(b,h(b))-d(h(b),h(a))
\\ & \geqslant & d(b,h(b))-d(a,b)
\\ & \geqslant & N_0-(2i(a,b)+1) \gg0
\end{eqnarray*}
where the second inequality follows from the fact that $h$ is an isometry, and the third inequality follows from Lemma \ref{intersection}.
\end{proof}

A Penner system of curves  on a surface $\Sigma=\Sigma_{g, 1}$ 
is a collection $C$  of $2g$ simple closed, essential, non  boundary parallel curves, that  splits into two disjoint sets $A$, $B$ with the following properties:
\begin{enumerate}
\item The curves  in each of $A=\{ a_1, \ldots, a_g \}$ and
$B=\{ b_1, \ldots, b_g\}$ are mutually disjoint.
\item We have $i(a_i, \ b_j)=1$ if  $i=j$ and zero otherwise.
\item $\Sigma$ cut along  $C=A\cup B$ is a $4g$-gon with a disk removed.
\end{enumerate}
\begin{lemma} \label{Pennerlarge} Let $\Sigma$ be a non-sporadic surface and let $h\in   \mathrm{Mod}(\Sigma)$ be a pseudo-Anosov mapping class. For every $N\gg0$, there is a Penner system of curves $C$ such that
for any $c \in C$ we have $d(c,\ h(c)) \geqslant N/2$. That is, there are systems $C$ for which the distance  $d(c,\ h(c))$ is arbitrarily large for any $c\in C$.
\end{lemma}
\begin{proof} We begin with the following claim that is attributed to Minsky in \cite{ColinHonda}.
\vskip 0.07in

{\it{Claim.}} Given $N\gg0$ we can find  a simple closed, essential, non-boundary parallel curve $c$ on $\Sigma$ such that  $d(c,\ h(c))=N$. 
\vskip 0.07in

{\it Proof of Claim.} We use the notation,  terminology and setting used, for instance in the proof of \cite[Proposition 4.6]{MaMi}. The pseudo-Anosov representative $h: \Sigma \to \Sigma$
acts on the space of projective measured laminations of $S$ and in there it determines exactly two fixed laminations $\mu$ and $\nu$. Now pick a geodesic lamination on $S$ that is minimal (i.e. it contains no proper sub laminations) and whose class, say $\lambda$, in the space of projective laminations is not $\mu$ or $\nu$. Now we have  $h(\lambda)\neq \lambda$.

Let  $\{a_n \in {\mathcal C}_0(\Sigma) \}_{n\in \NN}$ be a sequence that converges to $\lambda$ in the space of projective laminations. Then it is known that
$h(a_n)$ converges to $h(\lambda)$.

We claim that $d(a_n, h(a_n))\to \infty$ as  $n\to \infty$.  For otherwise, and after passing to a sub-sequence if necessary, we can assume that $d(a_n, h(a_n))=m<\infty$.
Now $a_n$ and $h(a_n)$ are connected by geodesics in the curve complex of $S$. Thus we can find a sequence $\{b_n \in {\mathcal C}_0(\Sigma) \}_{n\in \NN}$,
with  $d(b_n, h(a_n))=m-1$ and $d(b_n, a_n)=1$. Since $a_n$ is disjoint from $b_n$, after taking a subsequence if necessary, we can assume that 
$b_n$ converges to a class represented by a lamination, say $\lambda'$,  on $S$ that has intersection number  zero with $\lambda$, and $h(b_n)$
converges to a geometric lamination that has intersection number  zero with $h(\lambda)$. The minimality assumption implies that $\lambda'=\lambda$.
Thus $b_n$ converges to $\lambda$ and we have $d(b_n, h(b_n))\leqslant m-1$. Inductively, we will arrive at a sequence $\{c_n \in {\mathcal C}_0(\Sigma) \}_{n\in \NN}$
that converges to $\lambda$ and we have $d(c_n, h(c_n))=0$. This would imply that  $h(\lambda)= \lambda$.
which is a contradiction since we assumed that  $h(\lambda)\neq \lambda$. This finishes the proof of the claim.
\vskip 0.07in

To continue, with the construction of the desired Penner system of curves, start with $N\gg0$ and a curve $c$ with $d(c, h(c))=N$.
 If $c$ is non-separating set $a_1:=c$. Otherwise, since $g\geq 2$,  we can find a non separating essential non-boundary parallel curve $a_1$, with $i(a_1,  c)=0$.
Then $$N=d(c, h(c))\leqslant d(c, a_1)+ d(a_1,  h(a_1))+d(h(a_1), h(c)).$$
Since $d(c, a_1)=d(h(a_1), h(c))=1$ we obtain $d(a_1, h(a_1))\geqslant N-2$. 

Now pick $b_1$ so that $i(a_1, b_1)=1$ and complete $\lbrace a_1,b_1\rbrace$ into a Penner system of curves $C$.
Now for any $c \in C$ we have 
$$N-2\leqslant d(a_1, h(a_1))\leqslant d(a_1, c)+ d(c, h(c))+d(h(a_1),  h(c)).$$
Since $i(a, \ c)\leqslant 1$, by Lemma \ref{intersection},
$d(a_1, c)=d(h(a_1), h(c))\leqslant 3$. Thus $d(c, h(c))>N-8> N/2$ for $N\gg0$.
\end{proof}
\smallskip

\subsection{Stabilizing to pseudo-Anosov monodromies} 
\label{sec:stab_pA}
 For our applications to the AMU conjecture, we want the open-book decompositions we construct to have pseudo-Anosov monodromies, and we would like the genus of the fiber surface $F$ to take any arbitrary high value. In this subsection we discuss a technique to stabilize open book decompositions that preserve the pseudo-Anosov property of monodromies while producing  Stallings curves on the pages of the decomposition.
Our construction is inspired by an argument of Colin and Honda used in the proof of
\cite[Theorem 1.1]{ColinHonda}.

Let $(\Sigma, h)$ be an open book decomposition, with connected binding $\partial \Sigma,$ of an oriented, closed 3-manifold $M.$
Let $c$ be a simple closed, homotopically essential, non-peripheral curve $c$ on $\Sigma$ and let  $\epsilon$ be an embedded arc that runs from a point $P\in c$ to a point
on $\partial \Sigma$. Consider a closed neighborhood of $\epsilon$ on $\Sigma$, which is a rectangle with one side a closed neighborhood of  $P$, say $d \subset c$, containing $P$.
Replacing $d$ with the rest of the boundary of this rectangle leads to a curve that is isotopic to $c$ on $\Sigma$. This  new curve is the union of two properly embedded arcs $\gamma \cup \delta'$,
where $\delta'$ is an arc on $\partial \Sigma$ that connects the endpoints of $\gamma$ and contains the endpoint of $\epsilon$ on $\partial \Sigma$. See Figure \ref{fig:pushoff}. To facilitate our exposition we give the following definition.
\begin{figure}
\centering
\def \svgwidth{.37\columnwidth}
\begingroup%
  \makeatletter%
  \providecommand\color[2][]{%
    \errmessage{(Inkscape) Color is used for the text in Inkscape, but the package 'color.sty' is not loaded}%
    \renewcommand\color[2][]{}%
  }%
  \providecommand\transparent[1]{%
    \errmessage{(Inkscape) Transparency is used (non-zero) for the text in Inkscape, but the package 'transparent.sty' is not loaded}%
    \renewcommand\transparent[1]{}%
  }%
  \providecommand\rotatebox[2]{#2}%
  \newcommand*\fsize{\dimexpr\f@size pt\relax}%
  \newcommand*\lineheight[1]{\fontsize{\fsize}{#1\fsize}\selectfont}%
  \ifx\svgwidth\undefined%
    \setlength{\unitlength}{83.68098522bp}%
    \ifx\svgscale\undefined%
      \relax%
    \else%
      \setlength{\unitlength}{\unitlength * \real{\svgscale}}%
    \fi%
  \else%
    \setlength{\unitlength}{\svgwidth}%
  \fi%
  \global\let\svgwidth\undefined%
  \global\let\svgscale\undefined%
  \makeatother%
  \begin{picture}(1,0.67491485)%
    \lineheight{1}%
    \setlength\tabcolsep{0pt}%
    \put(0,0){\includegraphics[width=\unitlength,page=1]{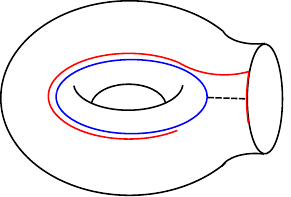}}%
    \put(0.38914356,0.54246345){\color[rgb]{1,0,0}\makebox(0,0)[lt]{\lineheight{1.25}\smash{\begin{tabular}[t]{l}$\gamma$\end{tabular}}}}%
    \put(0.38738249,0.40282809){\color[rgb]{0,0,1}\makebox(0,0)[lt]{\lineheight{1.25}\smash{\begin{tabular}[t]{l}$c$\end{tabular}}}}%
    \put(0,0){\includegraphics[width=\unitlength,page=2]{pushoff.pdf}}%
  \end{picture}%
\endgroup%

\caption{The curve $\gamma$ is obtained by pushing $c$ on $\partial \Sigma$ along the point $P,$ the endpoint of the dashed line on $c$.}
\label{fig:pushoff} 
\end{figure}
\begin{definition}\label{facilitate} With the setting as above, we will say that $\gamma$ is obtained by pushing $c$ on $\partial \Sigma$ along the point $P$. We will  also say that
the arc 
$\delta'\subset \partial \Sigma$ complements $\gamma$.
\end{definition}

We need the following.

\begin{proposition} \label{oneboundary} Let $\Sigma:=\Sigma_{g,1}$ and let $(\Sigma, \phi)$  be an open book decomposition of $M$ such that $\phi$ is pseudo-Anosov.
Suppose  that $\Sigma$ is non-sporadic and that $c$ is a  simple closed, homotopically essential, non-peripheral curve $c$ on $\Sigma$, such that  $d(c, \phi(c))=N$, for $N\gg0$. Let $\gamma$ be an arc
obtained by pushing $c$ on $\partial \Sigma$ along any point, and let  $(\Sigma', \phi')$  denote an open book decomposition obtained by stabilizing  $(\Sigma, \phi)$ along $\gamma$.
Then $\phi'$ is pseudo-Anosov.
\end{proposition}
\begin{proof}
 Let $c'$  the closed curve in $\Sigma'$ which  the core of the Hopf-band used in the stabilization, hence 
  $\gamma=c\cap \Sigma$.
 Recall that  $\phi'=\phi \circ \tau_{c'}^{\pm 1}$, according to whether we made a positive or a negative stabilization.
The proof of this proposition is essentially given in the proof of   \cite[Theorem 1.1]{ColinHonda}.
We note that although the statement of \cite[Theorem 1.1]{ColinHonda} contains the hypothesis that the starting monodromy $\phi$ is right-veering and the stabilization positive, this is actually only used to show that the new monodromy is also right-veering.

For the sake of completeness, we sketch the proof of \cite[Theorem 1.1]{ColinHonda}, while referring to their paper for full details.

 First we argue that $ \phi'=\phi \circ \tau_{c'}^{\pm 1}$ is not reducible, by arguing  that $\phi'(\delta)\neq \delta$ for any collection  $\delta$ of disjoint simple  curves  on $\Sigma'$.
  \vskip 0.06in
  
{\it Case 1.}  First, suppose that $\delta \subset \Sigma \subset \Sigma'.$ 
If  $i_{\Sigma}(\phi(\delta),c)=i_{\Sigma'}(\phi(\delta),c)=0,$  in which case $\delta$ can not be parallel to $\partial \Sigma$, 
then  $\phi'(\delta)=\phi(\delta) \neq \delta$ as $\phi$ is pseudo-Anosov. 
Suppose now that $i_{\Sigma}(\phi(\delta),c)=i_{\Sigma'}(\phi(\delta),c')=m>0.$ Then we claim that $i_{\Sigma'}(\phi'(\delta),a)=m$, where $a$ is the co-core of the stabilization band.
Indeed, there is an obvious representative of $\phi'(\delta)$ with $m$ intersection points with $a.$ If $i_{\Sigma'}(\phi'(\delta),a)<m,$ it means that there is a bigon consisting of a subarc of $a$ and a subarc of $\phi'(\delta).$ One checks that this bigon should be obtained by the gluing of a bigon in $\delta \cup c$ and a band that runs along $c',$ but no such bigon exists, otherwise we would $i_{\Sigma'}(\phi(\delta),c)<m;$ a contradiction.

\vskip 0.06in
{\it Case 2.} Assume that $\delta \nsubseteq \Sigma,$ that is $i_{\Sigma'}(\delta,a)=k>0.$ Let $B=\Sigma'\setminus \Sigma$ be the stabilization band, the intersection of $\phi(\delta)$ with $B$ consists of ``vertical arcs",  and arcs that intersect $c'.$ The later arcs may be either all of positive or all of negative slopes in $B.$  See \cite[Figures 1, 2]{ColinHonda}. Up to isotopy, we may assume that there is no triangle with boundary an arc of $c',$ an arc of $\delta$ and an arc of $\partial B.$ Under this hypothesis let $m$ be the number of arcs of $\delta \cap B$ of positive/negative slopes. Let $n$ be the number of other intersections of $c'$ and $\phi(\delta),$ so that $i_{\Sigma'}(\phi(\delta),c')=m+n.$
Then by a similar bigon chasing argument, Colin and Honda show that $i_{\Sigma'}(\phi'(\delta),a)=k\pm m+ n,$ where the sign depends on the sign of the stabilization and the sign of the slopes of the $m$ arcs. Thus if $\pm m + n\neq 0$ we have $\phi'(\delta) \neq \delta$ as $i_{\Sigma'}(\delta,a)=k \neq k\pm m+n=i_{\Sigma'}(\phi'(\delta),a).$ 

 We would like to stress that the above step is essentially the only place where allowing positive as well as negative stabilizations plays any role: it just exchanges the two cases where $i_{\Sigma'}(\phi'(\delta),a)=k+m+n$ and  $i_{\Sigma'}(\phi'(\delta),a)=k-m+n.$

In the case $\pm m+n=0,$ they show that $i_{\Sigma}(\delta,c')\neq i_{\Sigma'}(\phi'(\delta),c').$ Indeed, $i_{\Sigma'}(\phi'(\delta),c')\leqslant m+n\leqslant 2k.$ If $i_{\Sigma'}(\delta,c')\leqslant 2k,$ Colin and Honda use the fact that $d(c,\phi(c))\gg1$ and Lemma \ref{intersection} to deduce that $i_{\Sigma'}(\phi'(\delta),c')\gg2k,$ a contradiction. This finishes the proof that  $\phi'$ is not reducible.

Finally,  in the proof of \cite[Theorem 1.1]{ColinHonda}, it is shown  that $\phi'$ is not periodic by
considering the curve $\partial \Sigma$ on $\Sigma'$, and arguing that $i_{\Sigma'}((\phi')^n(\partial \Sigma),a) \underset{|n|\rightarrow \infty}{\longrightarrow} \infty.$
Note that, if $\phi'$  were periodic, then $(\phi')^m(\partial \Sigma)=\partial \Sigma$ for some $m\in \NN$ and thus above intersection number would be bounded.
 \end{proof}
\smallskip

Next we turn our attention to surfaces with two boundary components and stabilization along an arc that connects the two boundary components:
Suppose that $\Sigma=\Sigma_{g,2}$ and let $\beta$ be a properly embedded arc on $\Sigma$ connecting the two boundary components of  $\partial \Sigma$.
Thicken $\gamma$ into a rectangle $R$ that intersects the two components into sub-arcs of $\partial \Sigma$ that contain the endpoints of $\gamma$. The union of the complementary arcs
on $\partial \Sigma$, together with the  arcs of $\partial R$ that are parallel to $\gamma$,  give us a simple closed curve 
 $c_{\beta}\subset \Sigma$.

\begin{proposition} \label{twoboundary} Let $\Sigma:=\Sigma_{g,2}$ be non-sporadic and let $(\Sigma, \psi)$  be an  open book decomposition of $M$ such that $\psi$ is pseudo-Anosov.
Let $\beta$ be a properly embedded arc on $\Sigma$ connecting the two boundary components of  $\partial \Sigma$, and let $(\Sigma', \psi')$  denote an open book decomposition obtained by stabilizing  $(\Sigma, \phi)$ along $\beta$. If $d(c_{\beta}, \psi(c_{\beta}))=N$, for $N\gg0$,  then $\psi'$ is pseudo-Anosov.
\end{proposition}
\begin{proof} Let $c'$ be the closed curve in $\Sigma'$ which consist of the arc $\beta$ and the core of the Hopf-band coming from the stabilization.
Now $\psi'=\psi \circ \tau_{c'}^{\pm 1}$, according to whether we made a positive or a negative stabilization. The proof
also follows from  \cite[Subcase 2 of Theorem 1.1]{ColinHonda}. Again the statement of \cite[Theorem 1.1]{ColinHonda} contains the hypothesis that the starting monodromy $\psi$ is right-veering and the stabilization positive but a similar analysis as this in the proof of Proposition \ref{oneboundary} show that the conclusion holds without these hypotheses.
\end{proof}

Start with a genus $g$ open book decomposition $(\Sigma, h)$ of  $M$ such that  $h$ is pseudo-Anosov, and such that $\partial \Sigma$ has one component.
Let $C=A\cup B$ be a Penner system of curves on $\Sigma=\Sigma_{g, 1}$ where  $A=\{ a_1, \ldots, a_g \}$ and
$B=\{ b_1, \ldots, b_g\}$ as earlier.

 For $i=1,\ldots, g$, let $P_i$ denote the intersection point of $a_i$ and $b_i$. Let $\alpha_i$ and $\beta_i$ denote the arcs obtained by pushing $a_i$ and $b_i$
 on $\partial \Sigma$ along $P_i$. Let $\Gamma_C$ denote the set of these arcs.
 By Definition \ref{facilitate}, this process involves choosing arcs $\epsilon_i$ from $P_i$ to $\partial S$. We can choose these arcs so that they are mutually disjoint.
 For $i=1, \ldots, g$, we can arrange so that  we have,
 
 \begin{enumerate}
 
\item  the arcs in $\Gamma$ are pairwise disjoint
and
splits into two  sets $\{ \alpha_1, \ldots, \alpha_g \}$ and
 $\{ \beta_1, \ldots, \beta_g\}$.
 \item  the endpoints  $\alpha_i$   on $\partial \Sigma$ are separated by these
 of  $\beta_j$, precisely when $i=j$; 
 \item the intersection of the arcs complementing $\alpha_i$ and  $\beta_j$
 on  $\partial \Sigma$ is an interval when $i=j$ and empty otherwise.
 \end{enumerate}
 
 Now let $(\Sigma' , \  h_C)$ denote the open book decomposition obtained by  $(\Sigma, h)$ by a positive stabilization along each arc $\alpha_i$ and a negative stabilization along 
 $\beta_j$. Note that $\Sigma'$ has connected boundary and genus $2g,$ indeed, each pair of stabilizations  along the arcs $\alpha_i$ and $\beta$ leaves the boundary connected while it raises the genus by $1.$  For $g=2$ the situation is illustrated in Figure \ref{fig:stab_abelian}.
 
 \begin{theorem} \label{stabilize} Let  $(\Sigma, h)$  be a genus $g$ open book decomposition of $M$ with  $h$ pseudo-Anosov. 
 There is a Penner system of curves  $C$, 
 so that the monodromy of the stabilized open book decomposition  $(\Sigma' , \  h_C)$ is also  pseudo-Anosov.
 \end{theorem}
 
 \begin{proof} 

For any $N=4^g \ N_0\ggg$,  we can use Lemma \ref{Pennerlarge} to pick a  Penner system of curves   $C$ so that for any $c\in C$,
 $\displaystyle{d(c,  h(c))>\frac{ 4^{g}}{2}\ N_0}$.  Then we stabilize along the corresponding arcs in $\Gamma_C=\{\alpha_1,  \beta_1, \ldots, \alpha_g,  \beta_g\}$,
 by a positive stabilization along each arc $\alpha_i$ and a negative stabilization along each
 $\beta_j$.
 
The new monodromy $h_C$ is expressed as
 $$h_C=h \circ \tau_{a'_1} \circ \tau_{b'_1}^{-1}\circ \ldots \circ    \tau_{a'_g} \circ \tau_{b'_g}^{-1},$$
 where each curve contains the corresponding stabilization arc in $\Gamma_C$  and is  the core of the corresponding Hopf-band.
 
 Let us consider the stabilization process done in steps; one Hopf-band at a time. First, we stabilize along $\alpha_1$ to get an open book decomposition
  $(\Sigma_1, h_1)$, where  $h_1:=h \circ \tau_{a'_1} $ and $\Sigma_1$ has two boundary components.
By Proposition \ref{oneboundary} the monodromy  $h_1 $ is pseudo-Anosov.
By construction now $\beta_1$ has its endpoints on different components of $\partial \Sigma_1$.
Consider the curve $c_{\beta_1}$ constructed out of $\beta_1$ with the process described  before the statement of Proposition \ref{twoboundary}.
We claim  that we have
$$d(c_{\beta_1},\  h_1(c_{\beta_1}))>4^{g-1} \ N_0\gg g.$$ 
To see this claim first we note that, by construction, we have an arc $\delta^1\subset \partial \Sigma$ such that
(i)  $\delta^1$ connects the endpoints of $\beta_1$;
(ii) the interior of $\delta^1$ contains exactly one endpoint of $\alpha_1$;
(iii) the curve  $\beta_1\cup \delta^1$ is isotopic to $b_1$ on $\Sigma$.
We have $$d(b_1, \ h(b_1))\leqslant d(b_1, \ h_1(b_1)) + d( h_1(b_1), \ h(b_1)).$$
Using Lemma \ref{intersection}, we have  $d( h_1(b_1), \ h(b_1))=d( \tau_{a_1}(b_1), \ b_1) \leqslant 3$, since $i(a_1, \ b_1)=1$.
Thus $d(b_1, \ h_1(b_1))>\displaystyle {\frac{ 4^{g}}{2}\ N_0-3.}$

The intersection number of $c_{\beta_1}$  and $b_1$  on $\Sigma_1$ is $0$ (see Figure \ref{fig:stab_abelian}). Therefore
$$d(c_{\beta_1},h_1(c_{\beta_1}))\geqslant d(c_{\beta_1},b_1)-d(b_1,h_1(b_1))-d(h_1(b_1),h_1(c_{\beta_1}))>\frac{ 4^{g}}{2}\ N_0-5$$
and hence   $d(c_{\beta_1},\  h_1(c_{\beta_1}))>4^{g-1} \ N_0.$
We can apply Proposition \ref{twoboundary},  to conclude that $h_2=h_1\circ \tau^{-1}_{b'_1}$ is pseudo-Anosov. 

For $1<k \leq g$, let $(\Sigma_{k}, h_{k})$ be the genus $g+k$ open-book decomposition obtained by stabilizing on the first $k$ pairs of arcs $(\alpha_i,\beta_i)$ in  $\Gamma_C,$
where
$$h_{k}=h \circ \tau_{a'_1} \circ \tau_{b'_1}^{-1}\circ \ldots \circ    \tau_{a'_{k}} \circ \tau_{b'_{k}}^{-1}.$$
Proceeding inductively as in the case $k=1,$ we can show that for any $1\leqslant k \leqslant g,$ the map $h_{k}$
is pseudo-Anosov and that for any $c\in C$,
 $\displaystyle{d(c,  h_{k}(c))> 4^{g-k} \ N_0}\gg0,$ where the last distance is taken in ${\mathcal C}_1(\Sigma_{k}).$

Note that for $k=g,$ we have $h_g=h_C$  and $\displaystyle{d(c,  h_C(c))> \frac{N_0}{2}\gg0},$ for any $c\in C$.
 \end{proof}
 
 As a corollary of the proof of Theorem \ref{stabilize} we have the following which is a special case of Theorem 1.1 of \cite{ColinHonda}.

 \begin{corollary}
\label{thm:colin-honda}For any $3$-manifold $M,$ there exists a genus $g_0=g_0(M)$ such that for any $g\geqslant g_0,$ there is an open book decomposition $(\Sigma,h)$ of $M,$ such that $\partial \Sigma$ is connected, $\Sigma$ has genus $g,$ and $h$ is pseudo-Anosov.
\end{corollary}
\begin{proof}  As said earlier $M$ contains open book decompositions $(\Sigma,h)$, with connected binding and pseudo-Anosov monodromy $h$. Let $g_0=g_0(M)$ the smallest number that occurs as the genus of such an open-book decomposition. The result follows by induction applying Theorem \ref{stabilize}.
\end{proof}

The proof of Theorem \ref{stabilize} shows that if we let 
 $\Gamma_{k}=\{\alpha_1,  \beta_1, \ldots, \alpha_k,  \beta_k\}\subset \Gamma_{C} $, for $1\leq k\leq g$, the open book decomposition obtained by  stabilizing along  $\Gamma_{k}$ has
 pseudo-Anosov monodromy. More specifically, we have the following.

 \begin{corollary} \label{stabilizeless} Let  $(\Sigma, h)$  be a genus $g$ open book decomposition of $M$ such that $h$ is pseudo-Anosov.
 There is a Penner system of curves  $C$ so that, for $1\leq k\leq g$,  $(\Sigma_{k}, h_{k})$  has genus  $g+k$ 
 and the monodromy $h_{k}=h \circ \tau_{a'_1} \circ \tau_{b'_1}^{-1}\circ \ldots \circ    \tau_{a'_{k}} \circ \tau_{b'_{k}}^{-1}$
 is pseudo-Anosov.
 \end{corollary}

\subsection{Stallings twists and mapping class cosets}
\label{stab_stallings_twists}
Let $(\Sigma,h)$ be a genus $g\geq g_0$  open book decomposition of $M$ with $h$ pseudo-Anosov.
In subsection \ref{sec:stab_pA} we showed that
 we can choose a Penner system of curves
$C\subset \Sigma$  that leads to stabilized genus $g+k$ open book decomposition  $(\Sigma_{k}, h_{k})$, for any $1\leq k\leq g$.
Next we argue that the fibers of all these stabilized decompositions support Stallings twists.

\begin{lemma} \label{abelianStallings} 
For  $1\leq k\leq g$, the genus $g+k$ fiber $\Sigma_{k}$ contains at least $k$ non-boundary parallel and non-parallel to each other
Stallings curves.
\end{lemma}
\begin{proof}  
We will argue that   $k$ disjoint Stallings curves can be obtained on the stabilized  surface $\Sigma_k$ 
by taking appropriate band sums of the core curves of the stabilization bands.  

Note that while the curves $a_i$ and $b_i$ intersect in $\Sigma,$ any two of the stabilization arcs
sets $\{ \alpha_1, \beta_1 \ldots, \alpha_g, \beta_g\}$ are disjoint. Thus the cores of the stabilization bands
 $a_1', b_1' \ldots,a_k', b_k'$ are disjoint on $\Sigma_k,$ and they bound disjoint disks $E_1, E_1^* \ldots, E_k, E_k^*$ each intersecting the surface $\Sigma_k$ transversally.

As shown in Figure \ref{fig:stab_abelian}, we can choose disjoint arcs $\gamma_i,$ where $i=1,\ldots,k,$ so that $\gamma_i$ connects $a_i'$ to $b_{i+1}'$ ($\gamma_k$ connecting $a_k'$ to $b_1'$). Now the disks $D_i=E_i \#_{\gamma_i} E_{i+1}^*$ obtained by band sum along $\gamma_i$ for $i=1,\ldots ,k$ are mutually disjoint, and the boundary curves $t_i=\partial D_i$ satisfy $lk(t_i,t_i^+)=0.$ So the curves $t_1, \ldots, t_k$ are disjoint Stallings curves on $\Sigma_k.$ In $H_1(\Sigma_k,\partial \Sigma_k,\Z)$ they represent distinct non-zero elements, so they are non-parallel and non boundary parallel.

\begin{figure}
\centering
\def \svgwidth{.8\columnwidth}
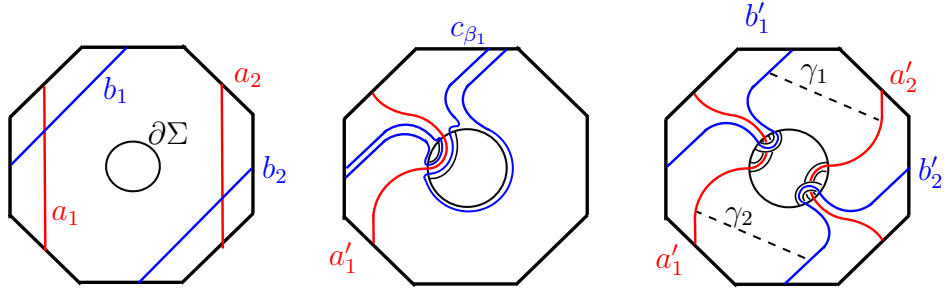
\caption{The surface $\Sigma$ is represented as a one-holed $4g$-gon (with $g=2$ here). From left to right: the Penner set of curves $a_i,b_i$, the curve $c_{\beta_1}$ on the surface stabilized along $\alpha_1,$ and the curves $a_i',b_i'$ and arcs $\gamma_i$ on the fully stabilized surface $\Sigma'.$}
\label{fig:stab_abelian} 
\end{figure}
\end{proof}

\vskip  0.07in
\begin{remark}\label{makedisjoint} Our construction of the Stallings curves in the proof of Lemma \ref{abelianStallings}  guarantees that if we think of the  Stallings curves $t_j$ as lying on different level fibers of
$M_{h_k}= \Sigma_{k}/_{(x,1)\sim ({ {h_k}(x),0)}},$ we can assure that the Stallings disks become disjoint in the  manifold $M$.
\end{remark}

\vskip  0.07in

We are now ready to prove Theorem \ref{thm:coset} stated in the Introduction. The construction of the desired cosets will be  done in two steps. In
Theorem \ref{prop:stab_abelian} below we construct abelian cosets and then in Theorem \ref{prop:stab_free} we deal with the construction of free cosets. As explained earlier,
the constructions become relevant to the AMU conjecture when applied to  q-hyperbolic manifolds.

\begin{theorem}\label{prop:stab_abelian} Let $M$ be a closed, oriented 3-manifold.  There is
 $g_0=g_0(M)$, so that 
 for any $g\geqslant g_0$, and any $1\leq k\leq g$, the following is true:
 Any genus $g$ open-book decomposition $(\Sigma,h)$ of $M$ with connected binding, and pseudo-Anosov $h$, 
 can be stabilized  to a genus $g+k$ decomposition $(\Sigma_k, h_k)$  such that $h_k$ represents a pseudo-Anosov element in $\mathrm{Mod}(\Sigma_{g+k,1})$. 
 Furthermore, there is an abelian subgroup  $A_k<\mathrm{Mod}(\Sigma_{g+k,1})$ such that,
 \begin{enumerate} 
 \item the rank of $A_k$ is $k$;
 \item  every element in the coset $h' A_k$ represents a  pseudo-Anosov mapping class; and
 \item  every element in $h' A_k$ occurs as monodromy  of an open book decomposition of $M.$
 \end{enumerate}
  \end{theorem}
\begin{proof}  Let $g_0$ be the genus guaranteed by Corollary  \ref{thm:colin-honda}, which can be taken to be the smallest number that occurs as genus of an open-book decomposition in $M$
with connected binding and pseudo-Anosov monodromy.
Now for any $g\geqslant g_0,$  there is an open book decomposition $(\Sigma,h)$ of $M,$ such that $\partial \Sigma$ is connected, $\Sigma$ has genus $g,$ and $h$ is pseudo-Anosov.
Now fix $g\geqslant g_0$ and apply  Theorem \ref{stabilize} and Corollary \ref{stabilizeless}  to  $(\Sigma, h)$. For any $1\leq k\leq g$,  we get $(\Sigma_k, h_k)$ of genus $g+k$ and $h_k$   pseudo-Anosov.
By Lemma  \ref{abelianStallings},  $\Sigma_{k}$ contains   a collection  $T=\{t_1, \ldots, t_k\}$ of disjoint Stallings curves. 
\vskip 0.07in

{\it {Claim.}} For any $t_i, t_j\in T$, we have $i_{\Sigma_k} (h_k(t_i), \ t_j)\neq 0$.

\vskip 0.07in

{\it {Proof Claim.}}  Let $C$ denote the Penner set of curves  that comes from the application of Theorem \ref{stabilize} and Corollary \ref{stabilizeless}.
By the argument  in the proof of Theorem \ref{stabilize} we conclude that $d(c, \ h_k(c))\gg0$, for any  $c\in C$.
 By the construction of the curves in $T$,  for any $t\in T$ and  $c\in C$ we have $i_{\Sigma_k}(t, \  c)\leqslant 1$. Hence,  by Lemma \ref{intersection} , we have $d(h_k(t), \  h_k(c))=d(t ,  c)\leqslant 3$ in the 
 curve graph of $\Sigma_k$.
 By triangle inequalities in this curve complex we get $$d(t, \  h_k(t)\geqslant d(c , \  h_k(c))-6\gg0.$$
 Since $i_{\Sigma_k}(t_i, t_j)=0$, for $t_i\neq t_j\in T$, we can now apply Lemma \ref{fill} to conclude that  $i_{\Sigma_k} (h_k(t_i), \ t_j)\neq 0$, which finishes the proof of the claim.
\vskip 0.07in

  Now we are in a situation where  Theorem \ref{hyperbolic} applies. Note that after putting the curves in $T$ at different level fibers in $M_{h_k}$, we can assure that the Stallings disks bounded by them are disjoint,
  as we noted in Remark \ref{makedisjoint}. 
 We apply Theorem \ref{hyperbolic} to conclude that
 there is $n\in \NN$ such that for all
 $(n_1, \ldots, n_k)\in {\Z}^k$ with all $|n_i|>n,$ the map
 $h_k \circ \tau^{n_1}_{t_1} \circ \ldots \circ \tau^{n_k}_{t_k}$ is pseudo-Anosov. Furthermore, the family 
 
 $$\{h_k\circ \tau^{n_1}_{t_1} \circ  \ldots \circ \tau^{n_k}_{t_k} , \ \ {\rm with} \ \ {(n_1, \ldots, n_k)\in {\Z}^k} \ \ {\rm and } \  \ |n_i|>n\}$$
 contains infinitely many pairwise  independent pseudo-Anosov mapping classes. Now taking $A_k$ to be the free abelian group generated by  $\tau^{n_i}_{t_i}$ for some $|n_i|>n$ we are done.
\end{proof}

Next we construct the free cosets promised earlier.

\begin{theorem}\label{prop:stab_free} Given a closed oriented 3-manifold $M$ and
 $g_0=g_0(M)$ as in Theorem \ref{prop:stab_abelian}, the following is true:
 For $g\geqslant \mathrm{max}(g_0,4)$,
any genus $g$ open-book decomposition $(\Sigma,h)$ of $M$, with connected binding, can be stabilized to a genus $g+4$ open book decomposition $(\Sigma', h')$ such that 
\begin{enumerate}
\item $h'$ represents a pseudo-Anosov mapping class in $\mathrm{Mod}(\Sigma_{g+4,1})$; and
\item there is a rank two free subgroup  $F<\mathrm{Mod}(\Sigma_{g+4,1})$  such that 
 every element in the coset $h' F$ occurs as monodromy  of an open book decomposition of $M.$
 \end{enumerate}
 \end{theorem}
  \begin{remark} Note that a pseudo-Anosov rank two free coset always contains infinitely many independent elements by Theorem \ref{hyperbolic}, as it always contains (many) pseudo-Anosov rank one Abelian cosets in particular. Actually, repeated application of Theorem \ref{hyperbolic} implies that "most"  elements, in the sense of Corollary \ref{more} and its proof, in the free coset are pseudo-Anosov.
 \end{remark}

\begin{proof} Recall that $g_0$ is the genus guaranteed by Corollary  \ref{thm:colin-honda}. Now fix  $g\geqslant \mathrm{max}(g_0,4),$
and apply  Corollary \ref{stabilizeless}  to obtain a genus $g+4$ open book decomposition $(\Sigma', h')$ with $h'$ pseudo-Anosov.
As in the proof of Lemma \ref{abelianStallings} , let $a_i', b'_i$, for $i=1, \ldots, 4$ denote the cores of the stabilization bands from $\Sigma$ to $\Sigma'$ which
 bound mutually disjoint disks $E_i, E^{*}_i$,
each intersecting the surface $\Sigma'$ transversally. Now form the band sums $D_a=E_1\#_{\gamma} E^{*}_{3}$ and $D_b=E_2\#_{\delta} E^{*}_{4}$
along  arcs $\gamma, \delta$ that intersect exactly once and run from $a_1$ to $b_3$ and from $a_2$ to $b_2$, respectively. The boundary curves
$a=\partial D_a$ and $b=\partial D_b$ are Stallings curves with their Stallings disks intersecting in a square on $\Sigma'.$ 
The intersection number of $a$ and $b$ is four.  See Figure  \ref{fig:stab_free}.

\begin{figure}
\centering
\def \svgwidth{.36\columnwidth}
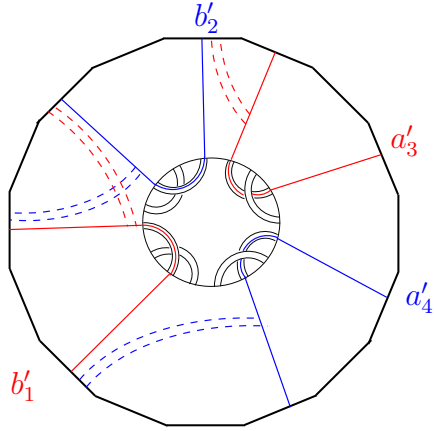
\caption{The choice of the arcs $\gamma, \delta$ illustrated in the case when $\Sigma$ has genus $4.$}
\label{fig:stab_free}
\end{figure}

Since $a,b$ have intersection four, the subgroup of  $\mathrm{Mod}(\Sigma)$ generated by Dehn twists on $a,b$ is a free group of rank 2. This is by  Theorem \ref{Hamidi-Tehrani}.
We claim that every  word $w\in F$ corresponds to an open book decomposition of $M$.
Indeed we have $w=\tau_a^{m_1}\tau_b^{n_1}\ldots \tau_a^{m_k}\tau_b^{n_k}$, for some $n_i, m_i\in \Z$. Take $k$ copies of $a\cup b$ and place then on different level fibers  of the open book decomposition as in the proof of Theorem \ref{hyperbolic}. Now for each of these $k$ copies place the copy of $b$ on a different fiber than the copy of $a$ that is slightly above the copy of $a$ and below the next copy of $a$. We can put the copies of the Stalllings disks now so that all the $2k$ curves bound disjoint disks. Since the open book to which   $fw=f\tau_a^{m_1}\tau_b^{n_1}\ldots \tau_a^{m_k}\tau_b^{n_k}$
corresponds is obtained by surgery of along the curves above, and since these curves bound disjoint disks the result follows. 
\end{proof}

Now we explain how to obtain Theorem \ref{thm:coset} stated in the introduction: For $g\geqslant \mathrm{max}(2g_0,4),$
Theorem  \ref{prop:stab_abelian} gives an  abelian coset of rank $\lfloor \frac{g}{2}\rfloor.$ and 
Theorems \ref{prop:stab_free} give a free coset. Thus for $g_1=\mathrm{max}(2g_0,4),$ we have 
Theorem  \ref{thm:coset}.


\section{Fibered knots in the Dehn fillings of the figure-eight knot}
\label{sec:fig8}
In this section, for $p\in \Z,$ let $4_1(p)$ be the $3$-manifold obtained by $p$-surgery on the figure eight knot. For $g$ a large enough integer, we will produce infinite families of hyperbolic fibered knots in some manifolds $4_1(p)$ with $|p|\geqslant 5$ that have fiber surface $\Sigma_{g,1}.$ Using  Theorems \ref{thm:q-hyperbolic} and \ref{tool}  we will assure that the  monodromies of these knots satisfy  the AMU conjecture and they give infinite families of pseudo-Anosov mapping classes that are independent in the sense 
defined in Section \ref{sec:independance}.
\subsection{A family of two-component hyperbolic fibered links in $S^3$}
\label{sec:family_links}
We start by considering  the family of two-component links $\{L_{l,m,k}\}_{l,m,k\geqslant 1}$ shown  in Figure \ref{fig:fig8example}.  As illustrated in the figure,
a box labeled $n$ denotes $n$ successive positive crossings between the two strings of the link the box involves. Similarly,   a double box labeled $k$ corresponds
 to stacking $k$-copies of
 pattern with two positive and two negative crossings on the corresponding  three strands  of the link.

\begin{figure}[!h]
\centering
\def \svgwidth{.6\columnwidth}
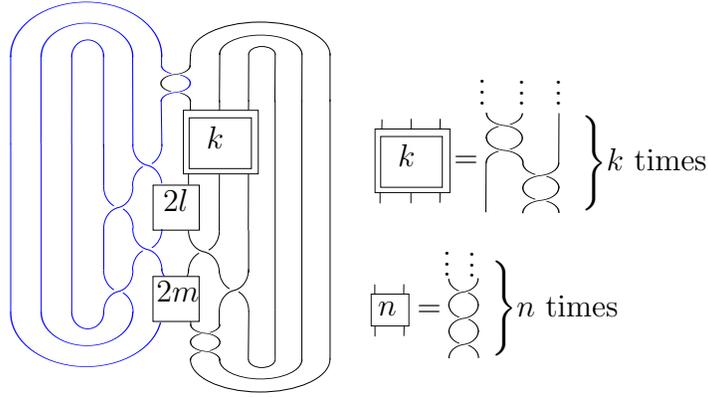
\caption{The links $L_{l,m,k}$.}
\label{fig:fig8example}
\end{figure}
The two-component link $L_{l,m,k}$ is the closure of an alternating braid. Furthermore, one of its two components, shown in blue in Figure \ref{fig:fig8example}, is a figure-eight knot.

We have the following.
\begin{lemma} \label{hyp}Let $l,m,k \geqslant 1.$
Then the complement of
 $L_{l,m,k}$ is hyperbolic and  fibers with fiber of genus $g(l,m,k)=2+m+l+2k$.
\end{lemma}
\begin{proof} By \cite{Stallings}, any alternating braid closure (actually, any homogeneous braid closure) is fibered. Thus the link $L_{l,m,k}$ is fibered. Moreover, one can concretely describe how a fiber surface is built: first take a disk for each strand of the braid (so, as many as the index of the braid) and connect them by twisted bands at crossings, where the orientation of the twisted bands agrees with the signs of the crossings. 

By this description, one can compute the Euler characteristic of the fiber surface $F_{l,m,k}$ of the closure of an homogeneous braid $\hat{\beta}$ by the formula $\chi(F_{l,m,k})=n(\beta)-c(\beta)$ where $n(\beta)$ and $c(\beta)$ are the braid index and number of crossings of $\beta.$ In the case of the $2$-component link $L_{l,m,k}$ of the figure, we get $$\chi(\Sigma)=6-(10+2m+2l+4k)=-4-2m-2l-4k.$$
 As $F_{l,m,k}$ has two boundary components, its genus is 
$$g(l,m,k)=-\chi(F_{l,m,k})/2=2+m+l+2k.$$

By a theorem of Menasco \cite{menasco:primediagram}, any prime non-split alternating diagram of a link that is not the standard diagram of the $T_{2,q}$ torus link represents a hyperbolic link. Here the diagram of $L_{l,m,k}$ in the figure is prime and non-split as long as both $l,m$ and $k$ are at least $1.$ \end{proof}

\begin{figure}
\centering
\def \svgwidth{.4\columnwidth}
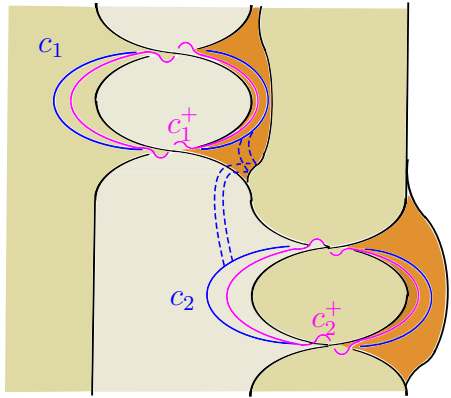
\caption{Creating Stallings twist curves on the fiber  $F_{l,m,k}$.}
\label{fig:Stallingstwist}
\end{figure}

\begin{lemma}\label{supportStallings} Let $L_{l,m,k}$ be a fibered hyperbolic link as in lemma \ref{hyp}. The fiber $F_{l,m,k}$ contains  $k$ disjoint, pairwise non-parallel and boundary non-parallel Stallings curves.
\end{lemma}
\begin{proof} Recall that we have $k$ blocks each of which contains a 3-braid with two positive and two negative crossings as illustrated in Figure \ref{fig:fig8example}. From each of these blocks we obtain a Stallings curve as follows: Consider a curve $c_1\subset F_{l,m,k}$ encircling the two negative crossings of the block and intersecting each co-core of the corresponding half-twisted bands exactly once; similarly consider
a curve $c_2\subset F_{l,m,k}$ encircling the two positive crossings. The curve $c_i$ ($i=1,2$) bounds an embedded disc $E_i$ in $S^3$ that intersects the $F_{l,m,k}$ transversally and we have $lk(c_1,c_1^+)=-1,$ 
while $lk(c_2,c_2^+)=1.$ For each block, we therefore construct a Stallings curve by connecting the curves $c_1,c_2$ by a band as shown in Figure \ref{fig:Stallingstwist}. The disk bounded by this curve $c$ is the band sum of $c_1$ and $c_2,$ and the curve satisfies $lk(c,c^+)=0.$ Note that the Stalligns curves corresponding to different block are disjoint, as they are supported on disjoint subsurfaces of $\Sigma:$ the subsurfaces neighboring the two positive and negative crossings of the block as in Figure \ref{fig:Stallingstwist}. 

Note that Stallings curves corresponding to different blocks are non-parallel, since each curve intersects the co-cores of the half-twisted bands of its block exactly once and  is disjoint from the  bands outside the block.
Furthermore, none of Stallings curves can be parallel to a component  of $\partial F_{l,m,k}$  since each of the two boundary components   intersects the co-cores of every half-twisted band in the $k$-blocks  an even number of times.
\end{proof}

\subsection{Abelian cosets from $S^3$} We now prove the following theorem which, in particular, for $m=4, l=1$ implies Theorem \ref{thm:fig8example} stated in the Introduction.

\begin{theorem}\label{thm:fig8examplegeneral}
Let $m,l, k\geqslant 1,$  with  $l+m\geqslant 5.$ In the mapping class group $\mathrm{Mod}(\Sigma_{2+m+l+2k,1}),$ there is a pseudo-Anosov abelian elementary coset of rank $k$ consisting only of mapping classes that satisfy the AMU conjecture. 
\end{theorem}

\begin{proof}  We begin by observing that the $k$ Stallings curves and disks of Lemma \ref{supportStallings}  are
disjoint from the figure-eight component of $L_{l,m,k}.$ Thus these disks will survive in any 3-manifold obtained by Dehn filling along the figure-eight component.
For $p \in \Z,$ let $4_1(p)$ be the manifold obtained by $p$ surgery on the figure-eight knot. We now want to perform a surgery on the figure-eight component of $L_{l,m,k},$ so that the resulting link will be a hyperbolic fibered knot $L'_{l,m,k}$ in a manifold $4_1(p).$

\begin{figure}
\centering
\def \svgwidth{.4\columnwidth}
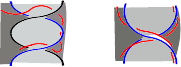
\caption{Crossing contributions  to the framing determined by $\gamma$.}
\label{fig:framing}
\end{figure}

Let $\gamma$ be the curve on the fiber surface $F_{l,m,k}$ obtained by pushing the figure-eight component inward along the fiber. If the slope $p$ chosen for the surgery agrees with the slope determined by $\gamma,$ then the link $L'_{l,m,k}$ in $4_1(p)$ will still be fibered. Moreover the fiber surface and the monodromy will now be obtained by simply capping off the relevant boundary component of the fiber surface of $L_{l,m,k},$ and the non-parallel Stallings curves will remain non-parallel Stallings curves on the new fiber.

To compute this framing we refer to Figure \ref{fig:framing} where the strands of the figure-eight component and the other component of  $L_{l,m,k}$ are colored by blue and black respectively and the curve $\gamma$ is shown in red. As illustrated  in the right side panel of the figure,  a  positive crossing (resp. negative crossing) between two strands of the figure-eight component contributes a $-1$ (resp. $+1$) to this framing. Furthermore,  as illustrated in the left side panel of the figure, two successive positive crossings between the figure-eight and the second component $L_{l,m,k}$ contribute a $-1.$
Hence 
the boxes of $2m$ and $2l$ successive positive crossings between the two-components of $L_{l,m,k}$ contribute $-m$ and $-l.$  In total,  since the self linking number of the blue component is zero, the framing determined by $\gamma$ is $-m-l.$

When $m+l\geqslant 5,$ the new link $L'_{l,m,k}$ is a link in a manifold $4_1(p)$ with $|p|\geqslant 5.$
 Thus we know that $lTV(4_1(p)\setminus L'_{l,m,k})\geqslant lTV(4_1(p)) >0$ by parts
(2) and (5)  of Theorem \ref{thm:q-hyperbolic}. That is the complement of the  link $L'_{l,m,k}$ in  $4_1(p)$ is q-hyperbolic.
The same is true for all the 3-manifolds obtained from the complements of $L'_{l,m,k}$ by
(i) drilling out  any number of the $k$ Stallings curves described above; or
(ii) performing any number of the Stallings twists along these $k$ curves. For all these manifolds
are also link complements in the same q-hyperbolic 3-manifold $4_1(p).$ 

These later links will be fibered in $4_1(p),$  and thus  by Theorem \ref{tool},
their  monodromies will satisfy the AMU conjecture.  However, at this stage, we do not know whether the monodromy is pseudo-Anosov; they might just have some pseudo-Anosov parts.
Next we will argue that this later case will not happen.
\vskip 0.07in

{\it Claim.} The link $L'_{l,m,k}$ is  hyperbolic and thus its monodromy is pseudo-Anosov.
\vskip 0.07in

 {\it Proof of Claim.} Recall that a twist region in a link diagram consists of a maximal string of bigons arranged end-to-end, so that the crossings alternate, and so that there are no other bigons adjacent to the ends; the twist number of
a diagram $D$ is the number of such twist regions, and is denoted by $t(D)$. For instance,
the twist number of the diagram of  $L_{l,m,k}$ in Figure \ref{fig:fig8example} is $10+2k$.
The reader is referred to \cite[Definition 2.4]{fkp:filling} for more details. 
A  criterion of Futer-Purcell \cite[Theorem 3.10]{futer-purcell}, states that, if a component $K$ of a hyperbolic link $L$ visits at least $7$ twist regions in a prime twist-reduced diagram of $L,$ then 
any non-trivial slope on the cusp of the link complement corresponding to the component $K$, has length bigger that 6. Then the ``6-Theorem" of Agol-Lackenby \cite[Theorem 3.13]{fkp:survey}
implies that all non-trivial fillings along the complement $K$ give hyperbolic manifolds. As the figure-eight component of $L_{l,m,k}$ does visit $7$ twist regions, we
conclude that  $L'_{l,m,k}$ is hyperbolic, finishing thereby the proof of the claim.  

\vskip 0.07in

Continuing with the proof of the theorem,  we can write the complement  of $L'_{l,m,k}$  as the mapping torus of a pseudo-Anosov $f\in \mathrm{Mod}(\Sigma_{2+m+l+2k,1})$.
Considering all the other monodromies we get after performing Stallings twists along the $k$ curves of Lemma \ref{supportStallings} we will get a pseudo-Anosov abelian elementary coset of rank $k$ of maps satisfying the AMU conjecture. More specifically, given $m,l \geqslant 1$,
let $H_k$ denote the free abelian group of $\mathrm{Mod}(\Sigma_{2+m+l+2k,1})$
generated by the $k$ Stallings twists on the $k$ curves. Now $fH_k$ is an elementary abelian  pseudo-Anosov  coset  such that all the mapping classes in it are realized as monodromies of  fibered links in $4_1(-l-m)$. This concludes the proof of  the theorem. \end{proof} 
\smallskip

Applying Theorem \ref{thm:fig8examplegeneral} for $m=4,l=1$ we get  
 the following  corollary which, by  Corollary \ref{infinitelymany}, gives Theorem \ref{thm:fig8example}  stated in the introduction.

\begin{corollary}\label{special} For any
$k\geqslant 1,$ there is a rank $k$ pseudo-Anosov abelian  elementary  coset of  $\mathrm{Mod}(\Sigma_{7+2k,1}),$  consisting of  mapping classes that satisfy the AMU conjecture.
\end{corollary}
 
 Let us point out that there is a lot of freedom in the construction described above and as a result we actually get many different pseudo-Anosov abelian elementary cosets. 
 For example we can always change the values of $m$ and $l$ while keeping $m+l\geqslant 5$ to get different families of links and monodromies.  In particular we have the following:
 
 \begin{corollary} \label{bounded}Fix $k\geqslant 1$. Then for every $g\geqslant  7+2k$, the  construction above gives   a pseudo-Anosov abelian elementary coset
 $fH_k$   in  $\mathrm{Mod}(\Sigma_{g,1}),$  such that
 \begin{enumerate}
 \item all mapping classes in $fH_k$ satisfy the AMU Conjecture; and
 \item for every pseudo-Anosov element in $fH_k$ the volume of the corresponding mapping torus is bounded above by a constant  that is independent of the genus.
 \end{enumerate}
  \end{corollary}
 \begin{proof} Fix $k\geqslant 1$. Consider the knots  $L'_{1,m,k}$ constructed in the proof of Theorem \ref{thm:fig8examplegeneral}, for $l=1$.
 As discussed above for $l=1$ and $m\geqslant 4$, the  knots $L'_{1,m,k}$  are hyperbolic and fibered in $4_1(-1-m)$, with fiber a surface of genus $g(m)=3+m+2k$ . 
 Note that  $g(m)\to \infty$ as $m\to \infty$.
 Given $m$, we obtain a coset $fH_k$ that satisfies part (1) of the statement of the corollary.  It remains to show that the volumes of the corresponding mapping tori are bounded as claimed.
 
 By construction, 
  the fibered manifolds with monodromies in $fH_k$
 are obtained by Dehn filling from $S^3$ along the link
  $J_{k}:=I_k \cup K$, where $K$ is the union of the disjoint Stallings curves on the fiber of the link $L_{1,m,k}$ and $I_k$  the link obtained by augmenting the twist region with the $m$ crossings with a crossing circle and then removing all the $2m$ crossings. Let $||S^3 \setminus J_k||$ denote the Gromov norm of the complement of $J_k$.  
The volume of any fibered manifold with monodromy in $fH_k$ will be bounded above by $B ||S^3 \setminus J_k||$  where $B$ is a universal constant.
By construction,
increasing $m$ leaves $B ||S^3 \setminus J_k||$ unchanged, as all the crossings in the $m$-box of Figure   \ref{fig:fig8example} lie in a single twist region, while 
as noted earlier it changes the genus $g(m):=3+m+2k$ of the link $L'_{1,m,k}$.
 \end{proof}

 We are now ready to prove the following, assuming Theorem \ref{thm:coset}, which we will prove in the next section.
 
 \begin{corollary} \label{independent} For any $k>0$ and $g\gg0$ we have two rank $k$ pseudo-Anosov abelian elementary cosets
  in $\mathrm{Mod}(\Sigma_{g,1})$ that satisfy the AMU conjecture and
 such that no conjugates of elements in one coset lies in the other.
 \end{corollary}
 
 \begin{proof} Given  $k>0$ and  $g\geqslant  7+2k$ take a coset  $fA$ given by Corollary \ref{bounded}. With the notation as in the proof of that corollary, let $M$ be a closed $q$-hyperbolic
 3-manifold with $||M||\gg B ||S^3 \setminus J_k||$ (compare with Corollary \ref{large}),
 and apply Theorem \ref{thm:coset} to $M$.
  Now for any $g\gg{\rm max}\{ g_1(M), \  7+2k\}$ take the second coset required by the statement of the corollary
 to be that given by Theorem \ref{thm:coset}. By Theorem \ref{tool} the mapping classes in this later coset, denoted by $f'H'$, satisfy the AMU  conjecture.
 Recall  that conjugate mapping classes  define homeomorphic mapping tori.
Since the Gromov norm of  mapping tori for elements in $f'H'$ are bounded below by $||M||$ 
 while mapping tori for elements in $fH$ are bounded above by $||M||$, the result follows.
 \end{proof}
 
 Finally to  obtain Corollary \ref{boundedintro}  stated in the Introduction consider the coset constructed in Corollary \ref{bounded} and take $C=C(k)= B ||S^3 \setminus J_k||,$  where this later quantity is defined in the proof of 
 Corollary \ref{bounded}.

\bibliographystyle{hamsplain}
\bibliography{biblio}
\end{document}